\documentclass[a4paper,12pt,reqno]{amsart}
\usepackage{fullpage}
\usepackage{amsmath,amssymb,graphicx,psfrag,hyperref}
\usepackage{color}
\usepackage{pgf,tikz}
\usepackage{pgfplots}
\usepackage{subfigure}
\usepackage{adjustbox}
\usepackage{moreverb}
\usepackage[procnames]{listings}
\usetikzlibrary{arrows}

\newcommand{\A}{\mathbb{A}}
\newcommand{\N}{\mathbb{N}}
\newcommand{\R}{\mathbb{R}}

\newcommand{\OO}{\mathcal{O}}

\newcommand{\MM}{\mathcal{M}}
\newcommand{\PP}{\mathcal{P}}
\renewcommand{\SS}{\mathcal{S}}
\newcommand{\TT}{\mathcal{T}}

\newcommand{\NN}{\mathcal{N}}

\newcommand{\capacity}{{\rm Cap}}
\newcommand{\ncap}{{\rm Cap}}
\newcommand{\refine}{{\tt refine}}

\newcommand{\dual}[3][]{#1\langle#2\,,\,#3#1\rangle}

\newcommand{\enorm}[2][]{#1|\!#1|\!#1|\,#2\,#1|\!#1|\!#1|}
\newcommand{\norm}[3][]{#1\|#2#1\|_{#3}}
\newcommand{\set}[3][\big]{#1\{#2\,:\,#3#1\}}

\newcommand{\Crel}{C_{\rm rel}}
\newcommand{\Ceff}{C_{\rm eff}}
\newcommand{\Cmark}{C_{\rm mark}}
\newcommand{\Clin}{C_{\rm lin}}
\newcommand{\qlin}{q_{\rm lin}}
\newcommand{\Copt}{C_{\rm opt}}

\newcommand{\err}{{\texttt{err}}}

\newtheorem{lemma}{Lemma}
\newtheorem{theorem}[lemma]{Theorem}
\newtheorem{proposition}[lemma]{Proposition}
\newtheorem{algorithm}[lemma]{Algorithm}
\newtheorem{remark}[lemma]{Remark}
\newtheorem{corollary}[lemma]{Corollary}

\newcommand{\revision}[1]{{\color{black}#1}}

\def\matrix#1{\mathbf{#1}}

\makeatletter
\def\@seccntformat#1{\hspace*{4mm}%
  \protect\textup{\protect\@secnumfont
    \ifnum\pdfstrcmp{subsection}{#1}=0 \bfseries\fi
    \csname the#1\endcsname
    \protect\@secnumpunct
  }%
}
\makeatother

\title{Adaptive boundary element methods for the computation of the electrostatic capacity on complex polyhedra}

\author{Timo Betcke}

\address{University College London, Department of Mathematics, Gower Street, London WC1E 6BT, UK}
\email{t.betcke@ucl.ac.uk}

\author{Alexander Haberl}
\author{Dirk Praetorius}

\address{TU Wien, Institute for Analysis and Scientific Computing, Wiedner Hauptstr.\ 8-10 / E101 / 4, 1040 Wien, Austria}
\email{\{Dirk.Praetorius\,,\,Alexander.Haberl\}@asc.tuwien.ac.at}

\date{\today}

\thanks{{\bf Acknowledgement.} The research of AH and DP is funded by the Austrian Science Fund (FWF) by the research project \emph{Optimal adaptivity for BEM and FEM-BEM coupling} (grant P27005)
and the special research program \emph{Taming complexity in PDE systems} (grant SFB F65).}


\begin{document}
\definecolor{uququq}{rgb}{0.25,0.25,0.25}
\definecolor{qqqqff}{rgb}{0,0,1}

\begin{abstract}
The accurate computation of the electrostatic capacity of three dimensional objects is a fascinating benchmark problem with a long and rich history. In particular, the capacity of the unit cube has widely been studied, and recent advances allow to compute its capacity to more than ten digits of accuracy. However, the accurate computation of the capacity for general three dimensional polyhedra is still an open problem. In this paper, we propose a new algorithm based on a combination of ZZ-type {\sl a posteriori} error estimation and effective operator preconditioned boundary integral formulations to easily compute the capacity of complex three dimensional polyhedra to 5 digits and more. While this paper focuses on the capacity as a benchmark problem, it also discusses implementational issues of adaptive boundary element solvers, and we provide codes based on the boundary element package Bempp to make the underlying techniques accessible to a wide range of practical problems.
\end{abstract}

\maketitle

\section{Introduction}
\label{section:intro}

\subsection{The capacity problem}

The capacity $\capacity(\Omega)$ of an isolated conductor $\Omega\subset\mathbb{R}^3$
measures its ability to store charges. It is defined 
as the ratio of the total surface equilibrium charge
relative to its surface potential value \cite{kellogg}. To compute the capacity,
we therefore need to solve the following exterior Laplace problem for the equilibrium
potential $u$ with unit surface value:
\begin{subequations}
\label{eq:laplace}
\begin{align}
-\Delta u &= 0& &\hspace*{-20mm}\text{in }\mathbb{R}^3 \backslash \overline \Omega,\label{eq:laplace1}\\
u &= 1& &\hspace*{-20mm}\text{on }\Gamma:=\partial\Omega,\label{eq:laplace2}\\
|u(x)| &= \mathcal{O}\left(|x|^{-1}\right)& &\hspace*{-20mm}\text{as }|x|\rightarrow\infty\label{eq:laplace3}.
\end{align}
\end{subequations}
The total surface charge of an isolated conductor is then given by Gauss' law as
\begin{equation}
\label{def:capacity}
\capacity^*(\Omega)=-\epsilon_0\int_{\Gamma}\frac{\partial u}{\partial\nu}(x)\,d\Gamma(x).
\end{equation}
Here, $\nu(x)$ is the outward pointing normal vector for $x\in\Gamma$, and $\epsilon_0$ is the electric constant with value $\epsilon_0\approx 8.854\times 10^{-12}\,{\rm F/m}$. In the rest of the paper, we will use the normalized
capacity $\ncap(\Omega)=-\frac{1}{4\pi}\int_{\Gamma}\partial u/\partial\nu\,d\Gamma$, which is commonly used in the literature.

\subsection{State of the art}

Analytic expressions of the capacity in 3D are only known for very few simple domains, such as a sphere with radius $r$, for which 
$\ncap(\Omega)=r$. Moreover, computing the capacity to high accuracy even for simple shapes such as the unit cube is exceedingly difficult as it involves the solution of the exterior Laplace problem~\eqref{eq:laplace} for (possibly) non-smooth domains. This is very different from the 2D case, where techniques such as fast Schwarz-Christoffel maps \cite{driscoll02,banjai03} allow the computation of the logarithmic capacity to many digits of accuracy even on complex domains.

Nevertheless, there have been a range of interesting developments over the years to compute the capacity in three dimensions. A frequently used benchmark example is the capacity of the unit cube, which we will also use in this paper to compare our results to existing methods. An early bound for the normalized capacity $C$ of the unit cube was published by P\'olya already in 1947 \cite{polya47} who estimated that
$$
0.62033 < C < 0.71055.
$$
Over the years, several improvements have been made, some of which are summarized in  \cite{HMW10}. In that paper, techniques based on random walks are used to compute the capacity as  $0.66067813$ with an error believed to be in the order of  $\pm 1.01\times 10^{-7}$.
In~2013,~\cite{hp13} improved the existing computations by using Nystr\"om methods and multilevel solvers to $0.66067815409957$ with an error in the order of $10^{-13}$.

Adaptive boundary element computations for the capacity are not completely new. For example, in \cite{Read97} a pre-chosen anisotropic refinement towards the edges is used together with an extrapolation technique to compute the capacity of the unit cube to around six digits of accuracy. However, the computations in that paper are simplified by exploiting the special symmetry of the cube and do not generalize to arbitrary polyhedra. Moreover, \cite{Read97} briefly mentions also an adaptive refinement strategy that is based on refining elements with large charge contributions.

\subsection{Contributions of the present work}

In this paper, we present a black-box method for capacity computations of polyhedra in three dimensions, which achieves a similar accuracy of order $10^{-6}$ for the unit cube. Our method is based on an adaptive boundary element computation, which uses a ZZ-type {\sl a~posteriori} error estimator to steer the mesh-refinement in combination with a suitable operator preconditioning.
Our ansatz is completely generic in that the adaptive refinement strategy works for any type of polyhedron and quickly generates meshes that compute the capacity for a given shape to several digits of accuracy. 

\def\diam{{\rm diam}}
\subsection{Main results in a nutshell}

A boundary integral formulation for the computation of the capacity can be derived by considering Green's representation theorem and noting that the exterior Laplace double-layer potential is zero for constant densities. We hence obtain a representation of the solution $u$ of \eqref{eq:laplace} as
\begin{equation}
\label{eq:slp}
u(x) = \int_{\Gamma} G(x,y)\phi(y)\,d\Gamma(y) \quad\text{for all } x\in\mathbb{R}^3\backslash \overline{\Omega},
\end{equation}
where $\phi:=-{\partial u}/{\partial\nu}$ is the (negative) normal derivative of the exterior solution $u$ and $G(x,y):=\frac{1}{4\pi|x-y|}$ is the  Green's function of the 3D Laplacian. Taking boundary traces, the right-hand side of~\eqref{eq:slp} gives rise to  the Laplace single-layer integral operator $V$, and we arrive at the boundary integral equation of the first kind
\begin{equation}
\label{eq:laplace_integral}
1 = \int_{\Gamma} G(x,y)\phi(y)\,d\Gamma(y) =: [V\phi](x)\quad\text{for all }x\in\Gamma.
\end{equation}
Then,
\begin{align}
	\capacity(\Omega) = \frac{1}{4 \pi} \int_{\Gamma} \phi \, d\Gamma.
\end{align}
The challenge is to accurately compute the solution $\phi$ of~\eqref{eq:laplace_integral} close to edges and corners of a polyhedron. High-order collocation or Nystr\"{o}m methods work well for smooth obstacles, but are difficult to apply to surfaces with corners or edges. Therefore, we propose to use a Galerkin boundary element method (BEM) combined with
{\sl a~posteriori} error estimation and adaptive mesh-refinement. The proposed algorithm is of the common type
$$
\boxed{\tt~solve~}\quad\longrightarrow\quad
\boxed{\tt~estimate~}\quad\longrightarrow\quad
\boxed{\tt~mark~}\quad\longrightarrow\quad
\boxed{\tt~refine~}
$$
and generates a sequence of successively refined triangulations $\TT_\ell$ and corresponding Galerkin approximations $\Phi_\ell\approx\phi$ such that $\ncap_\ell(\Omega) := \frac{1}{4\pi}\int_\Gamma \Phi_\ell\,d\Gamma$ converges 
to $\ncap(\Omega)$ at the optimal algebraic rate. The mathematical foundation of such algorithms for BEM has recently been derived in~\cite{fkmp,gantumur,fkmp:part1,fkmp:part2}. One novelty in the present paper is that we combine adaptivity with an effective preconditioning strategy.

We assume that $\TT_\ell$ is a triangulation of $\Gamma$ into plane surface triangles. Let $\TT_\ell^{\rm dual}$ be  the induced dual grid; see Figure~\ref{fig:dual_element}. Let $\SS^1(\TT_\ell)$ be the $\TT_\ell$-piecewise affine globally continuous functions and $\PP^0(\TT_\ell^{\rm dual})$ be the $\TT_\ell^{\rm dual}$-piecewise constant functions on $\Gamma$. It is known \cite{StWe98,Hiptmair2006} that one can use the (regularized) discrete hypersingular integral operator $D_\ell^{\rm reg}$ on $\SS^1(\TT_\ell)$ to effectively precondition the discrete weakly-singular integral operator $V_\ell^{\rm dual}$ on $\PP^0(\TT_\ell^{\rm dual})$. The Galerkin approximation $\Phi_\ell^{\rm dual} \in \PP^0(\TT_\ell^{\rm dual})$ is obtained as 
\begin{equation}\label{eq:laplace_precond}
\Phi_\ell^{\rm dual} = D_\ell^{\rm reg} \widetilde{\Phi}_{\ell},
\quad \text{where $\widetilde{\Phi}_{\ell}$ solves the preconditioned system} \quad
(V_\ell D_\ell^{\rm reg} \widetilde{\Phi}_{\ell}) = 1.
\end{equation}
Having computed $\Phi_\ell^{\rm dual}$, we use that each vertex $z$ of $\TT_\ell$ corresponds to precisely one dual cell $T^{\rm dual}_z \in \TT_\ell^{\rm dual}$. To steer the mesh-refinement, we compute the ZZ-type error indicators
\begin{align}\label{eq:zz:T}
 \eta_\ell(T) := \diam(T)^{1/2} \, \norm{\Phi_\ell^{\rm dual} - I_\ell\Phi_\ell^{\rm dual}}{L^2(T)} 
 \quad \text{for all } T \in \TT_\ell,
\end{align}
where $I_\ell\Phi_\ell^{\rm dual} \in \SS^1(\TT_\ell)$ is the unique function with $I_\ell\Phi_\ell^{\rm dual}(z) = \Phi_\ell^{\rm dual}|_{T^{\rm dual}_z}$ for all vertices $z$ of $\TT_\ell$.
We refer to~\cite{zz1987,rodriguez1994,bc2002} for ZZ-type estimators for FEM \revision{and to~\cite{ffkp14,FFHKP15} for BEM.}
The local contributions $\eta_\ell(T)$ are then used to mark elements for refinement. An improved triangulation $\TT_{\ell+1}$ is obtained from $\TT_\ell$ by refining (essentially) these marked elements.

The mentioned results on adaptive BEM~\cite{fkmp,gantumur,fkmp:part1,fkmp:part2} consider residual error estimators, which provide
more mathematical structure than the heuristical ZZ-error estimator~\eqref{eq:zz:T}. However, we note that the evaluation and integration of the BEM residual
is usually more costly than the computation of the BEM solution ( see Remark~\ref{remak:costs} below), while the ZZ-error estimator comes essentially at no cost.

The striking advantage of the proposed strategy is that both discrete integral operators $D_\ell^{\rm reg}$ as well as $V_\ell^{\rm dual}$ can effectively be treated as follows: Let $\TT_\ell^{\rm bary}$ be the barycentric refinement of $\TT_\ell$. We then build the discrete weakly-singular integral operator $V_\ell^{\rm bary}$ with respect to $\PP^0(\TT_\ell^{\rm bary})$. Both operators $D_\ell^{\rm reg}$ and $V_\ell^{\rm dual}$ can be obtained by (sparse) projection operators applied to $V_{\ell}^{\rm bary}$; see Section~\ref{subsection:preconditioning}. Since the computation~\eqref{eq:laplace_precond} of $\Phi_\ell^{\rm dual}$ by iterative solvers relies only on matrix-vector products, we altogether assemble and store only $V_\ell^{\rm bary}$. Finally, since the number of elements satisfies $\#\TT_\ell^{\rm bary} = 6 \,\#\TT_\ell$, preconditioning and adaptivity do not lead to any significant overhead of the overall method.



\begin{figure}
	\centering
	\adjincludegraphics[width=0.50\textwidth,Clip={.10\width} {0.0\height} {0.1\width} {.0\height}]{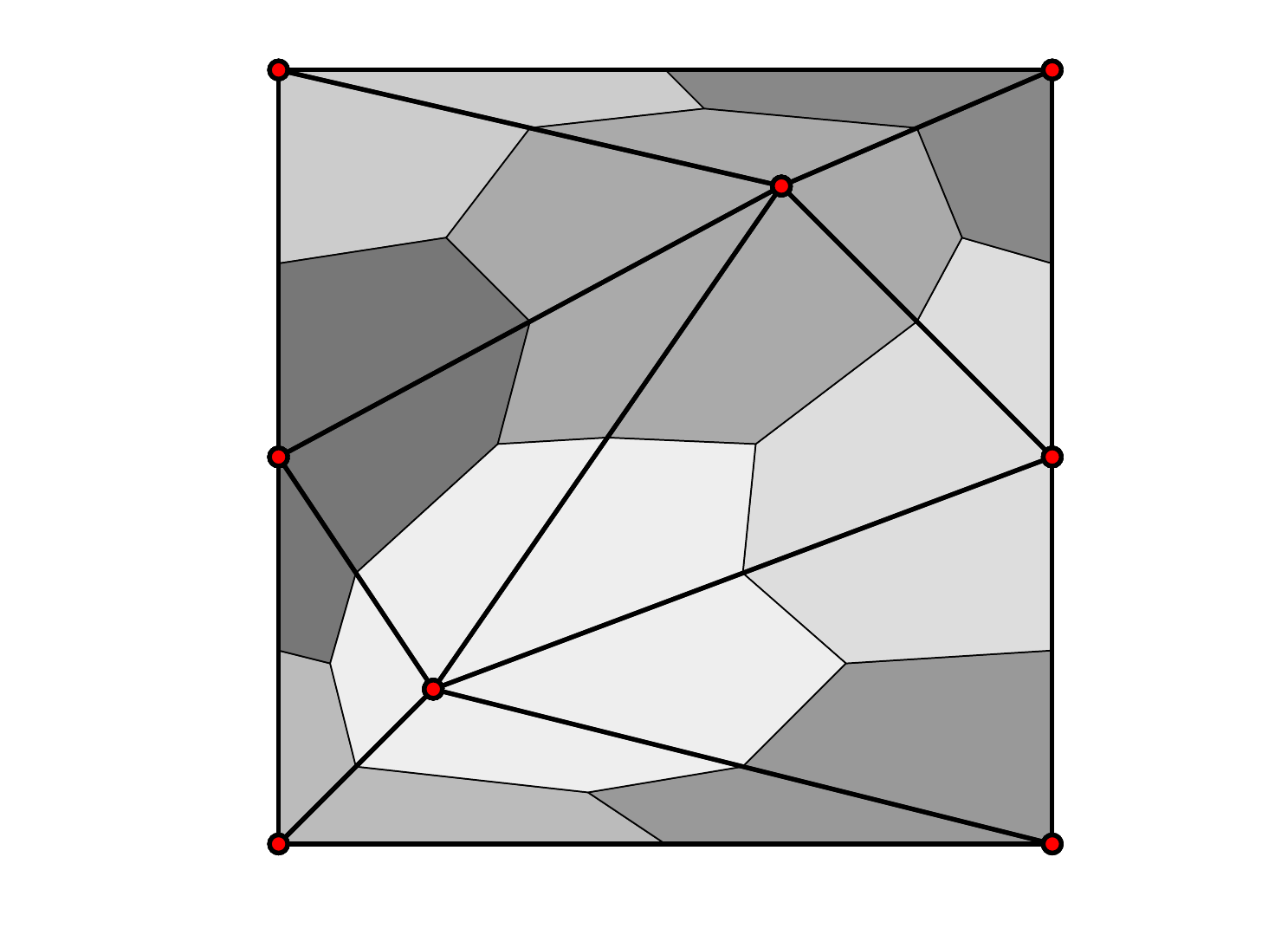}
	\adjincludegraphics[width=0.50\textwidth,Clip={.10\width} {.\height} {0.10\width} {.0\height}]{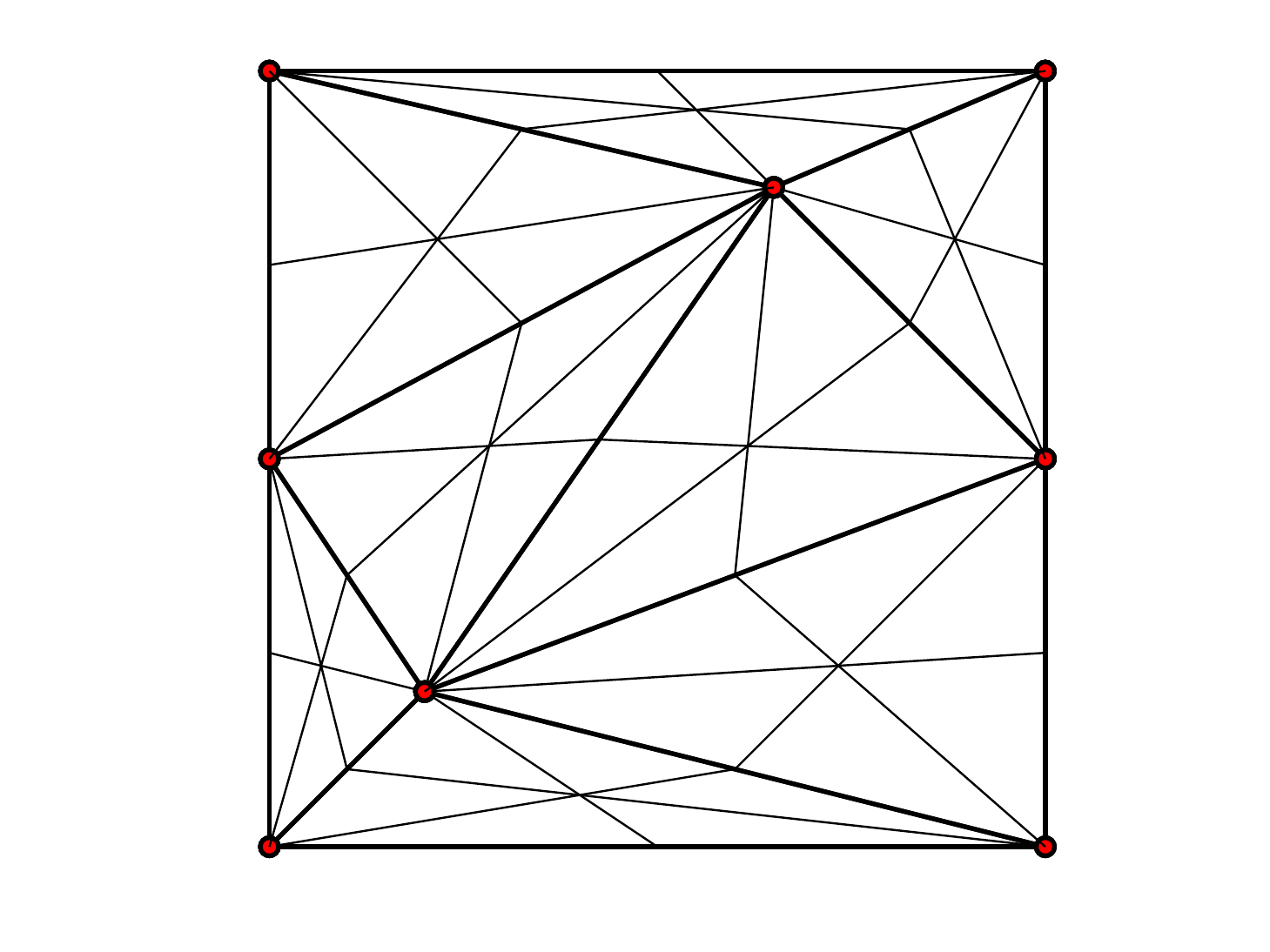}
\caption{Construction of the dual mesh $\TT_\ell^{\rm dual}$ (left) and the barycentric refinement $\TT_\ell^{\rm bary}$ (right) for a given mesh $\TT_\ell$ indicated by the thick edges (left and right). 
Each (polygonal) dual element $T^{\rm dual} \in \TT_\ell^{\rm dual}$ (left, gray) corresponds to one node of $\TT_\ell$. Each (triangular) element $T^{\rm bary} \in \TT_\ell^{\rm bary}$ belongs to precisely one triangle $T \in \TT_\ell$.}
\label{fig:dual_element}
\end{figure}


\begin{figure}
\center
\includegraphics[width=7.5cm]{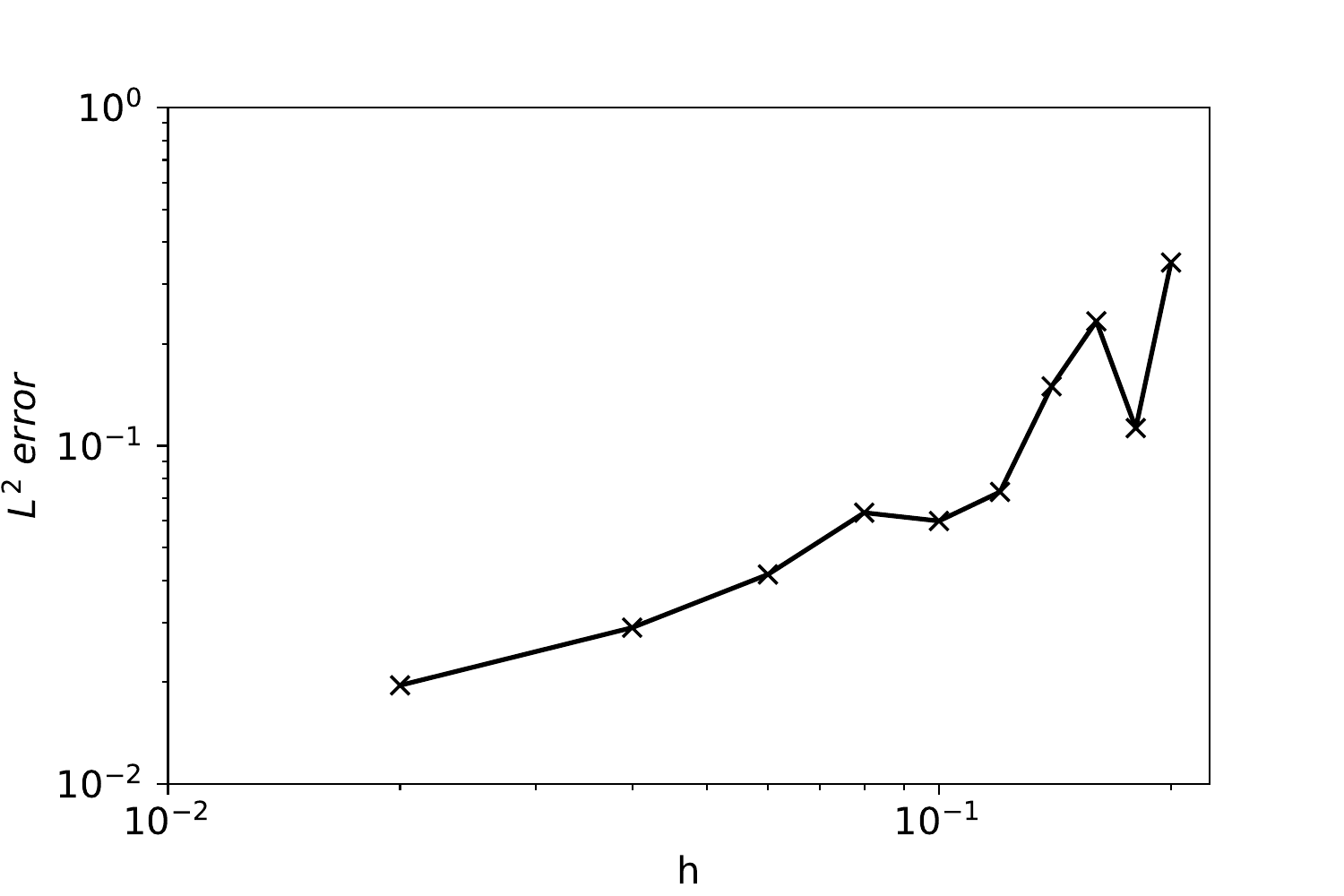}
\includegraphics[width=7.5cm]{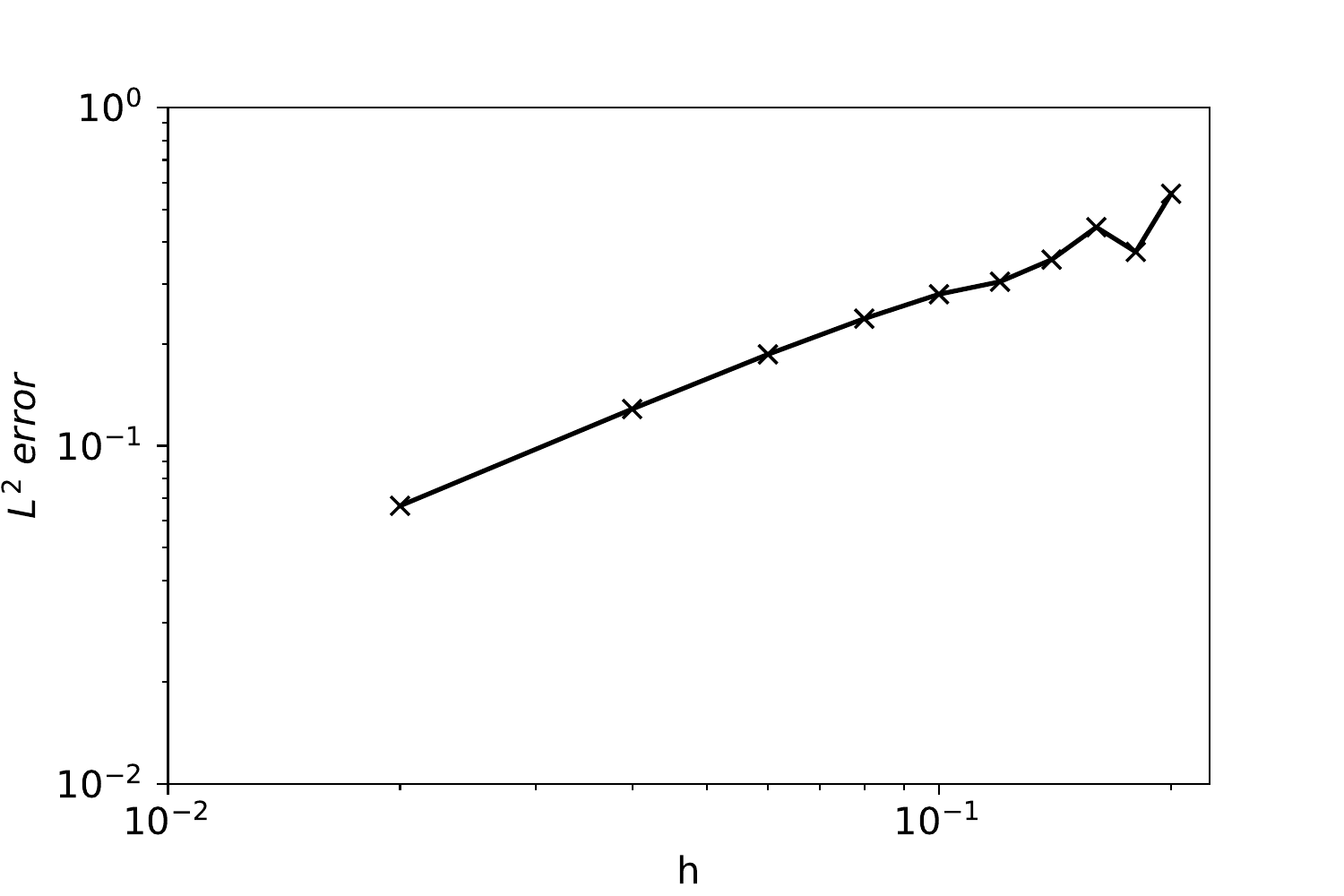}
\caption{Approximation error of the BEM for a Laplace--Dirichlet problem on the unit cube with continuous, piecewise linear basis functions on the primal mesh (left) and piecewise constant basis functions on the dual grid (right).}
\label{fig:cube_convergence}
\end{figure}

\section{Preconditioned adaptive BEM for Laplace problems}
\label{section:abem}

\noindent%
In this section, we describe the numerical implementation of the preconditioned boundary integral equation \eqref{eq:laplace_precond} and the corresponding ZZ-type error estimator. This forms the basis for our adaptive capacity algorithm described in Section \ref{section:capacity}.

The standard way of approximating Neumann data on polyhedral boundaries is the approximation through piecewise constant basis functions on the primal grid, given the discontinuity of Neumann data cross edges. Here, we choose a different approximation space, namely the space of piecewise constant functions on the dual grid. The advantage is that this space admits a stable duality pairing with continuous, piecewise affine functions on the primal grid, allowing us to use efficient operator preconditioning techniques based on the duality of the single-layer and hypersingular boundary operator, while at the same time making it possible to apply a ZZ-type error estimator as shown below.

Moreover, even though the piecewise constant basis on the dual grid is locally continuous across edges, it recovers optimal $h$-uniform convergence to solutions of Laplace--Dirichlet problems. An example is given in Figure~\ref{fig:cube_convergence}, which demonstrates the $h$-con\-ver\-gence in the dual basis to the Neumann data of a Laplace--Dirichlet problem on the unit cube with analytically known solution. The shown error is the relative $L^2$-error for the piecewise smooth Neumann data. The convergence is linear.

While a complete convergence analysis in the dual basis is beyond the merit of this paper, a simple heuristic argument is the following: The dual basis is continuous in a radius of diameter $h$ across an edge, and discontinuous globally. Hence, locally around an edge we have an error contribution that decreases with order $h$, while not propagating beyond due to the discontinuity of the basis. We therefore expect to recover the optimal convergence order $h$ for the $L^2$-error.

\begin{figure}[]
	 \centering
	 \psfrag{T0}{}
	 \psfrag{T1}{}
	 \psfrag{T2}{}
	 \psfrag{T3}{}
	 \psfrag{T4}{}
	 \psfrag{T12}{}
	 \psfrag{T34}{}
	 \includegraphics[width=0.9\textwidth]{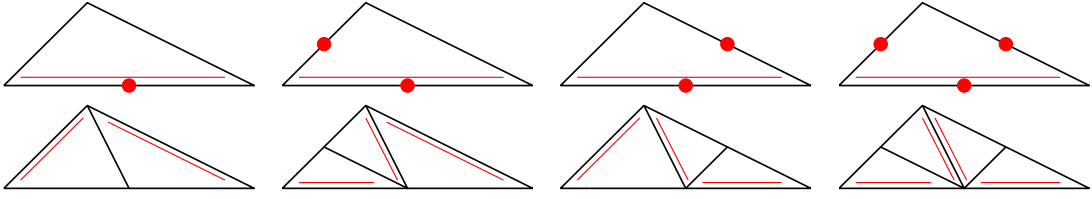} \quad
	 \caption{For each triangle $T \in \TT_\ell$, there is one fixed \textsl{reference edge},
	 indicated by the double line (left, top). Refinement of $T$ is done by bisecting
	 the reference edge, where its midpoint becomes a new vertex of the refined
	 triangulation $\TT_{\ell+1}$. The reference edges of the son triangles are opposite to
	 this newest vertex (left, bottom). To avoid hanging nodes, one proceeds as
	 follows: We assume that certain edges of $T$, but at least the reference edge,
	 are marked for refinement (top). Using iterated newest vertex bisection,
	 the element is then split into 2, 3, or 4 son triangles (bottom).}
	 \label{fig:nvb}
\end{figure} 
%
\subsection{Triangulations}

Throughout, let $\Omega\subset\mathbb{R}^3$ be a bounded polyhedral domain with closed boundary $\Gamma$. We consider conforming triangulations $\TT_{\ell} = \{T_{\ell,1}, \dots, T_{\ell,N_\ell}\}$ of $\Gamma$ into $N_{\ell} = \#\TT_\ell$ plane (closed) surface triangles $T_{\ell,j} \in \TT_{\ell}$. Let $\NN_\ell = \{z_{\ell,1}, \dots, z_{\ell,M_\ell}\}$ be the set of vertices of $\TT_\ell$, which contains the $M_\ell = \#\NN_\ell$ vertices. 

To obtain the dual mesh $\TT_\ell^{\rm dual}$, we connect the center of gravity of each element $T_{\ell,j} \in \TT_\ell$ with the midpoints of its edges. These lines define $M_\ell$ non-degenerate closed polygons $T_{\ell,j}^{\rm dual}$, which are collected in the \emph{dual mesh} $\TT_\ell^{\rm dual} = \{T_{\ell,1}^{\rm dual}, \dots, T_{\ell,M_\ell}^{\rm dual} \}$, where $\#\TT_\ell^{\rm dual} = M_\ell = \#\NN_\ell$. Note that each cell $T_{\ell,j}^{\rm dual} \in \TT_\ell^{\rm dual}$ contains precisely one node $z_{\ell,j} \in \NN_\ell$ in its interior, and we use the according numbering of $\TT_\ell^{\rm dual}$ and $\NN_\ell$, i.e., $z_{\ell,j} \in {\rm interior}(T_{\ell,j}^{\rm dual})$ for all $j = 1, \dots, M_\ell$; see Figure~\ref{fig:dual_element} (left).

Finally, we refine each element $T_{\ell,j} \in \TT_\ell$ into six triangles by connecting the center of gravity of $T_{\ell,j}$ with the midpoints of its edges as well as its vertices. This gives rise to the \emph{barycentric refinement} $\TT_\ell^{\rm bary} = \{T_{\ell,1}^{\rm bary}, \dots T_{\ell,6N_\ell}^{\rm bary}\}$. We note that $\#\TT_\ell^{\rm bary} = 6 \, \#\TT_\ell = 6 \, N_\ell$. Moreover, each element $T_{\ell,j} \in \TT_\ell$ and $T_{\ell,j}^{\rm dual} \in \TT_\ell^{\rm dual}$ is the union of elements in $\TT_\ell^{\rm bary}$, i.e.,
\begin{align*}
 T = \bigcup \set{T^{\rm bary} \in \TT_\ell^{\rm bary}}{T^{\rm bary} \subset T}
 \quad \text{for all } T \in \TT_\ell \cup \TT_\ell^{\rm dual};
\end{align*}
see Figure~\ref{fig:dual_element} (right). In other words, $\TT_{\ell}^{\rm bary}$ is the coarsest common refinement
of $\TT_\ell$  and $\TT_{\ell}^{\rm dual}$ into plane surface triangles.

\subsection{Discrete function spaces}

In this section, we collect the discrete spaces of functions $\Gamma \to \R$, which are employed below.

On $\TT_\ell^{\rm dual}$, we consider the space $\PP^0(\TT_\ell^{\rm dual})$ of $\TT_\ell^{\rm dual}$-piecewise constant functions. We choose the basis $\{\chi_{\ell,1}^{\rm dual}, \dots, \chi_{\ell,M_\ell}^{\rm dual}\}$ consisting of characteristic functions, i.e., $\chi_{\ell,j}^{\rm dual}(x) = 1$ on $T_{\ell,j}^{\rm dual}$ and zero otherwise.

On $\TT_\ell$, we consider the space $\SS^1(\TT_\ell)$ of $\TT_\ell$-piecewise affine functions, which are globally continuous.
We choose the usual nodal basis $\{\varphi_{\ell,1}, \dots, \varphi_{\ell,M_\ell}\}$, which consists of the hat functions characterized by $\varphi_{\ell,i}(z_{\ell,i}) = 1$ and $\varphi_{\ell,i}(z_{\ell,j}) = 0$ for $i \neq j$.

On $\TT_\ell^{\rm bary}$, we consider the space $\PP^0(\TT_\ell^{\rm bary})$ of $\TT_\ell^{\rm bary}$-piecewise constant functions. 
Again, we choose the basis $\{\chi_{\ell,1}^{\rm bary}, \dots, \chi_{\ell,6N_\ell}^{\rm bary}\}$ consisting of characteristic functions, i.e., $\chi_{\ell,j}^{\rm bary}(x) = 1$ on $T_{\ell,j}^{\rm bary}$ and zero otherwise.

\subsection{Galerkin discretization of the Laplace integral equation}

For general $f\in H^{1/2}(\Gamma)$ (and $f=1$ for the capacity problem), we consider the weakly-singular boundary integral equation  
\begin{align}\label{eq:strongform}
 V\phi(x) := \frac{1}{4\pi}\int_\Gamma \frac{\phi(y)}{|x-y|} \, d\Gamma(y)
 = f(x)
 \quad\text{for all } x\in\Gamma,
\end{align}
where $V:H^{-1/2}(\Gamma)\rightarrow H^{1/2}(\Gamma)$ is the Laplace single-layer integral operator.

We refer, e.g., to the monographs~\cite{mclean,hsiao-wendland,steinbach,sauter-schwab,eps_gwinner} for the following facts on the functional analytic framework:

Let $\dual{g}{\psi}_{\Gamma} := \int_{\Gamma} g(x) \psi(x) \, d\Gamma(x)$ be the $L^2(\Gamma)$-based duality pairing between functions $g \in H^{1/2}(\Gamma)$ and $\psi \in H^{-1/2}(\Gamma)$. We can reformulate \eqref{eq:strongform} as variational problem in the form
\begin{align}\label{eq:weakform}
 \dual{V\phi}{\psi}_\Gamma = \dual{f}{\psi}_\Gamma
 \quad\text{for all }\psi\in H^{-1/2}(\Gamma).
\end{align}
From the ellipticity of $V$ in $H^{-1/2}(\Gamma)$ with respect to $\dual{\cdot}{\cdot}_{\Gamma}$, the unique solvability of~\eqref{eq:weakform} follows by the Lax--Milgram lemma. In particular, we note that $\enorm{\psi} := \dual{V\psi}{\psi}_\Gamma^{1/2} \simeq \norm{\psi}{H^{-1/2}(\Gamma)}$ defines the equivalent energy norm on $H^{-1/2}(\Gamma)$.
Here and throughout, $\lesssim$ abbreviates $\le$ up to some generic multiplicative constant, and $\simeq$ abbreviates that both estimates $\lesssim$ and $\gtrsim$ hold.

Since $\PP^0(\TT_\ell^{\rm dual}) \subset H^{-1/2}(\Gamma)$, we can discretize~\eqref{eq:weakform} by 
a Galerkin discretization~\cite{steinbach,sauter-schwab}: Find $\Phi_\ell^{\rm dual} \in \PP^0(\TT_\ell^{\rm dual})$ such that
\begin{align}\label{eq:discreteform}
 \dual{V\Phi_\ell^{\rm dual}}{\Psi_\ell^{\rm dual}}_\Gamma = \dual{f}{\Psi_\ell^{\rm dual}}_\Gamma
 \quad\text{for all }\Psi_\ell^{\rm dual} \in \PP^0(\TT_\ell^{\rm dual}).
\end{align}%
Again, unique solvability of~\eqref{eq:discreteform} follows by the Lax--Milgram lemma. 
Moreover, the latter is equivalent to the linear system of equations
\begin{equation}
\label{eq:laplace_system}
\matrix{V}_{\ell}^{\rm dual}\matrix{x}_{\ell} = \matrix{f}_{\ell}^{\rm dual}
\end{equation}
with the symmetric and positive definite matrix $\matrix{V}_{\ell}^{\rm dual} \in \mathbb{R}^{M_\ell\times M_\ell}_{\rm sym}$ and given right-hand side $\matrix{f_{\ell}}^{\rm dual} \in \R^{M_\ell}$ defined as $\matrix{V}_{\ell}^{\rm dual}[i, j] = \dual{V\chi_{\ell,j}^{\rm dual}}{\chi_{\ell,i}^{\rm dual}}_{\Gamma}$ and $\matrix{f}_\ell^{\rm dual}[i] = \dual{f}{\chi_{\ell,i}^{\rm dual}}_{\Gamma}$. Then, the unique solution $\Phi_\ell^{\rm dual} \in \PP^0(\TT_\ell^{\rm dual})$ of~\eqref{eq:discreteform} satisfies
$\Phi_{\ell}^{\rm dual} = \sum_{j} \matrix{x}_{\ell}[j] \chi_{\ell,j}^{\rm dual}$, 
where $\matrix{x}_{\ell} \in \R^{M_\ell}$ is the unique solution to the linear system~\eqref{eq:laplace_system}.

\subsection{ZZ-type \textsl{a~posteriori} error estimation}
\label{section:ZZ}

Having computed the Galerkin solution $\Phi_\ell^{\rm dual} \in \PP^0(\TT_\ell^{\rm dual})$, we define $I_\ell \Phi_\ell^{\rm dual} := \sum_{j} \matrix{x}_{\ell}[j] \varphi_{\ell,j} \in \SS^1(\TT_\ell)$ as well as the ZZ-type error indicators~\eqref{eq:zz:T} for all $T \in \TT_\ell$.
For model problems with known solution, our numerical experiments led to the evidence that 
\begin{align}
 \sum_{T \in \TT_\ell} \eta_\ell(T)^2
 = \norm{h_\ell^{1/2}(1-I_\ell)\Phi_\ell^{\rm dual}}{L^2(\Gamma)}^2
 \simeq \enorm{\phi - \Phi_\ell^{\rm dual}}^2
 \simeq \norm{\phi - \Phi_\ell^{\rm dual}}{H^{-1/2}(\Gamma)}^2,
\end{align}
where $h_\ell : \Gamma \to \R$ is the $\TT_\ell$-piecewise mesh-size function $h_\ell |_T = \diam(T)$. A thorough mathematical proof goes beyond of the scope of the present work. We note that the efficiency estimate $\lesssim$ can be obtained with the scaling techniques from~\cite{ffkp14}, while the more important reliability estimate $\gtrsim$ remains open.

\subsection{A generic adaptive strategy}

With the ZZ-type error estimator at hand, we consider the 
following generic adaptive algorithm, which locally adapts the primal mesh $\TT_\ell$, while the Galerkin approximations $\Phi_\ell^{\rm dual}$ are computed on the corresponding dual mesh $\TT_\ell^{\rm dual}$. For the local mesh-refinement, we employ 2D newest vertex bisection (NVB); see, e.g.,~\cite{kpp13,stevenson07} and Figure~\ref{fig:nvb} for some illustration. We note that NVB provides rich mathe\-matical structure and, in particular, guarantees that all adaptive meshes are conforming and uniformly shape regular (i.e., generated triangles $T$ cannot deteriorate).

\begin{algorithm}\label{algorithm}
\textbf{Input:} Initial conforming triangulation $\TT_0$ of $\Gamma$ into plane surface triangles, adaptivity parameters $0<\theta\le1$ and $\Cmark\ge1$.
\\
\textbf{Adaptive loop.} Iterate the following steps~{\rm(i)}--{\rm(iv)} for all $\ell=0,1,2,\dots$:
\begin{itemize}
\item[\rm(i)] Build the dual mesh $\TT_\ell^{\rm dual}$ and compute the solution $\Phi_\ell^{\rm dual} \in \PP^0(\TT_\ell^{\rm dual})$ of~\eqref{eq:discreteform}.  
\item[\rm(ii)] Compute the local contributions $\eta_\ell(T)$ from~\eqref{eq:zz:T} for all $T\in\TT_\ell$.
\item[\rm(iii)] Determine a set of, up to the multiplicative constant $\Cmark$, minimal cardinality, which satisfies the D\"orfler marking criterion~\cite{doerfler96}
\begin{align}\label{eq:doerfler}
 \theta\,\sum_{T \in \TT_\ell} \eta_\ell(T)^2 \le \sum_{T\in\MM_\ell}\eta_\ell(T)^2,
\end{align}
i.e., the local contributions associated with $\MM_\ell$ control a fixed percentage of $\eta_\ell^2$.
\item[\rm(vi)] Use NVB to create the coarsest conforming triangulation $\TT_{\ell+1} :=\refine(\TT_\ell,\MM_\ell)$, where all marked elements $T\in\MM_\ell$ have been bisected.
\end{itemize}
\textbf{Output:} Sequence of successively refined triangulations $\TT_\ell$ as well as corresponding Galerkin solutions $\Phi_\ell^{\rm dual} \in \PP^0(\TT_\ell^{\rm dual})$
and ZZ-type error estimators $\eta_\ell$.\qed
\end{algorithm}

The code of a Python implementation of Algorithm~\ref{algorithm} based o the Bempp library is given in Appendix~\ref{section:app:code}.
Similar algorithms driven by computationally expensive residual error estimators have been mathematically analyzed in recent years~\cite{gantumur,fkmp,fkmp:part1,fkmp:part2,abem+solve}. 
We refer to the appendix for details on available results.

\begin{remark}
	\begin{itemize}		
	\item[\rm(i)] In practice, the linear system~\eqref{eq:laplace_system} is solved iteratively. In our implementation, (the preconditioned) GMRES solver
	is stopped, if the iterates $\Phi^{\rm dual}_{\ell,k} \approx \Phi^{\rm dual}_{\ell}$ satisfy that 
	$\enorm{\Phi^{\rm dual}_{\ell,k-1} -\Phi^{\rm dual}_{\ell,k}} \leq \lambda \eta_\ell(\Phi^{\rm dual}_{\ell,k})$ where $\lambda \approx 10^{-3}$. 
	Here, $ \eta_\ell(\Phi^{\rm dual}_{\ell,k})$ denotes the ZZ-error estimator~\eqref{eq:zz:T} evaluated at the inexact solution $\Phi^{\rm dual}_{\ell,k}$
	instead of the exact Galerkin solution $\Phi^{\rm dual}_{\ell}$. We refer to the recent work~\cite{abem+solve} for a thorough mathematical analysis of the stopping criterion.
	
	\item[\rm(ii)]
	Since the BEM matrix $V_\ell^{dual}$ is symmetric, it would be possible
	to employ a preconditioned CG algorithm (PCG) for the iterative solution of~\eqref{eq:laplace_system}. However, generic
	implementations of PCG are restricted to symmetric preconditioners. Since we employ an
	operator preconditioning with unsymmetric mass matrix (see Section~\ref{subsection:preconditioning} below) and since the
	BEM++ library~\cite{bempp} also treats non-symmetric BEM formulations, our implementation uses a generic implementation of precondioned GMRES.
    
	\end{itemize} 
\end{remark}

\begin{remark}
\label{remak:costs}
In the numerical experiments below, we compare our adaptive algorithm (with BEM solution $\Phi_\ell^{\rm dual} \in \PP^0(\TT_\ell^{\rm dual})$ with a standard adaptive algorithm (with BEM solution $\Phi_\ell \in \PP^0(\TT_\ell)$ on the primal mesh) driven by a residual error estimator. In order to compute the weighted-residual error estimator for lowest order BEM, one can do the following: 
Instead of assembling up the discrete single-layer operator $\matrix{V}_\ell^{\PP^0}$ on $\PP^0(\TT_\ell)$ , 
we assemble the operator $\matrix{V}_\ell^{\PP^1}$ on $\PP^1(\TT_\ell)$. 
By applying sparse projection operators onto $\matrix{V}_\ell^{\PP^1}$, we obtain the matrix $\matrix{V}_\ell^{\PP^0}$ as well as the evaluation the residual in $\PP^1(\TT_\ell)$
to compute the residual error estimator. 
Since the change of discrete spaces from $\PP^1(\TT_\ell)$ to $\SS^1(\TT_\ell)$ increases the degrees of freedom by a factor $3$, 
the computational costs grow at least by the same factor.  

On the other hand, the proposed Algorithm~\ref{algorithm} with the ZZ-type error estimator requires to 
assemble the discrete single-layer operator $\matrix{V}_\ell^{{\rm bary}}$ on $\PP^0(\TT_\ell^{\rm bary})$, which increases the degrees of freedom by a factor $6$, 
but already includes the operator preconditioning; see Section~\ref{subsection:preconditioning} for further details. 
Hence, using a suitable implementation with FMM or $\mathcal{H}$-Matrices, Algorithm~\ref{algorithm} with an built-in operator preconditioning
just leads to double assembling costs compared to a non-preconditioned adaptive scheme based on the residual error estimator.

\end{remark}%

\subsection{Operator preconditioning with the hypersingular operator}
\label{subsection:preconditioning}

While the integral equation \eqref{eq:laplace_system} can be solved via dense LU decomposition for smaller system sizes, larger problems require iterative methods, in particular, if BEM acceleration techniques such as $\mathcal{H}$-matrices~\cite{bebendorf2008,boerm2010,hackbusch2015} or FMM~\cite{gr1997,osw2006,gr2009} are used. An efficient preconditioning strategy is based on operator preconditioning with the hypersingular operator. We recall from the literature~\cite{mclean,hsiao-wendland} that the hypersingular operator $D: H^{1/2}(\Gamma)\rightarrow H^{-1/2}(\Gamma)$ is the negative normal derivative of the double-layer boundary operator, i.e.,
$$
\left[Dv\right](x) := -\frac{\partial}{\partial \nu(x)}\int_{\Gamma}\frac{\langle \nu(y), x-y\rangle}{4\pi \, |x-y|^3} \, v(y) \, d\Gamma(y).
$$
For $v, w\in H^{1/2}(\Gamma)$, it follows from integration by parts that~\cite{steinbach,sauter-schwab,eps_gwinner}
$$
\langle Dv, w\rangle = \frac{1}{4\pi}\int_{\Gamma}\int_{\Gamma}\frac{\langle \text{\underline{curl}}_{\Gamma}v(y), \text{\underline{curl}}_{\Gamma}w(x)\rangle}{|x-y|}\, d\Gamma(y).
$$
In case of the Laplace equation and $\Gamma$ being connected, the hypersingular operator has a one-dimensional nullspace consisting of the constant functions on $\Gamma$. In order to obtain a suitable preconditioner, we therefore introduce the regularized operator
$$
\left[D^{\rm reg} v\right](x) := \left[Dv\right](x) + \int_{\Gamma}v(x)\, d\Gamma(x).
$$
This operator shifts the zero eigenvalue for constant functions to the value $|\Gamma|$ and otherwise does not change the spectrum.
In~\cite{StWe98}, it was shown that the regularized hypersingular operator is an effective preconditioner for the single-layer boundary operator. We use the regularized hypersingular operator as a right preconditioner in \eqref{eq:laplace_system} and in addition, an inverse mass matrix as a left preconditioner to arrive at
\begin{equation}
\label{eq:laplace_precond_discrete}
\matrix{M}_\ell^{-\top} \matrix{V}_{\ell}^{\rm dual} \matrix{M}_{\ell}^{-1} \matrix{D_{\ell}^{\rm reg}} \matrix{\tilde{x}}_{\ell} = \matrix{M}_\ell^{-\top} \matrix{f}_{\ell}^{\rm dual}
\end{equation}
with $\matrix{x_{\ell}} = \matrix{M}_\ell^{-1}\matrix{D_{\ell}^{\rm reg}}\matrix{\tilde{x}_{\ell}}$, 
$\matrix{D}_{\ell}^{\rm reg}[i, j] = \dual{D_{\ell}^{\rm reg}\phi_{\ell, j}}{\phi_{\ell, i}}$, $\matrix{M_{\ell}}[i, j]^{\top} = \dual{\phi_{\ell, j}}{\chi_{\ell, i}^{\rm dual}}$, and the other quantities defined as before.
The implementation of~\eqref{eq:laplace_precond_discrete} only requires the assembly of the matrix $\matrix{V}_{\ell}^{\rm bary} \in \R^{6N_\ell \times 6N_\ell}$ associated with the single-layer boundary operator 
on the space 
$\mathcal{P}^{0}(\mathcal{T}_{\ell}^{\rm bary})$,
i.e., $\matrix{V}_{\ell}^{\rm bary}[i,j] = \dual{V\chi_{\ell,j}^{\rm bary}}{\chi_{\ell,i}^{\rm bary}}_\Gamma$. The matrices $\matrix{V}_{\ell}^{\rm dual}$ and $\matrix{D_{\ell}^{\rm reg}}$ can then be obtained through simple sparse projection operators. To obtain $\matrix{V}_{\ell}^{\rm dual}$, we notice that the basis functions $\chi_{\ell,j}^{\rm dual}$ in $\mathcal{P}^{0}(\mathcal{T}_{\ell}^{\rm dual})$ can be written as a linear combinations of the basis $\chi_{\ell,j}^{\rm bary}$ in $\mathcal{P}^{0}(\mathcal{T}_{\ell}^{\rm bary})$. Let $\matrix{P}_{\ell} \in \R^{6N_\ell \times M_\ell}$ be the sparse matrix that maps coefficients of basis functions in $\mathcal{P}^{0}(\mathcal{T}_{\ell}^{\rm dual})$ to a vector of coefficients of the associated basis functions in $\mathcal{P}^{0}(\mathcal{T}_{\ell}^{\rm bary})$. Then,
$$
\matrix{V}_{\ell}^{\rm dual} = \matrix{P}_{\ell}^\top \matrix{V}_{\ell}^{\rm bary} \matrix{P}_\ell.
$$
Similarly, we recognize that we can write
$$
\matrix{D_{\ell}^{\rm reg}} = \sum_{m=1}^3\matrix{Q}_{\ell, m}^\top\matrix{V}_{\ell}^{\rm bary}\matrix{Q}_{\ell, m},
$$
where $\matrix{Q}_{\ell, m} \in \R^{6N_\ell \times M_\ell}$ maps coefficient vectors in $\SS^1(\mathcal{T}_{\ell})$ to the piecewise constant $m$-th component of the corresponding surface curl operator, which can be represented as a linear combination of basis functions in $\mathcal{P}^0(\mathcal{T}_{\ell}^{\rm bary})$. 

Putting everything into \eqref{eq:laplace_precond_discrete}, we obtain the discrete system
$$
\matrix{M_\ell^{-\top}}\matrix{P}_{\ell}^\top\matrix{V}_{\ell}^{\rm bary} \matrix{P}_\ell\matrix{M_{\ell}^{-1}}\sum_{m=1}^3\matrix{Q}_{\ell, m}^\top\matrix{V}_{\ell}^{\rm bary}\matrix{Q}_{\ell, m}\matrix{\tilde{x}}_{\ell} = \matrix{M_\ell^{-\top}}\matrix{f}_{\ell}^{\rm dual}.
$$
Hence, we only require the assembly of the operator $\matrix{V}_{\ell}^{\rm bary}$ and four matrix-vector products with this operator for each evaluation of the left-hand side. Assuming a fast method with a linear or log-linear complexity for assembly and matrix-vector product, the total cost of the assembly is therefore six times as high as that of the non-preconditioned system and each matrix-vector product about 24 times as expensive as that for a non-preconditioned system. However, we will see later that this preconditioner is very effective and only a small number of matrix-vector products will be required in our numerical experiments. Moreover, the error estimation essentially comes for free as part of this preconditioning strategy; see Section~\ref{section:ZZ}.

The associated sparse matrix operations with $\matrix{P}_{\ell}$ and $\matrix{Q}_{\ell}$ are cheap and negligible compared to the evaluation of the integral operators. Using a modern sparse direct solver, also the cost of the LU decomposition of the sparse mass matrix $\matrix{M}_{\ell}$ and its application is reasonably cheap.

We note that the analysis of~\cite{StWe98} is restricted to quasi-uniform meshes. However, we refer to~\cite{hju2014} for operator preconditioning for 2D BEM on graded meshes.
The extension of the latter result to 3D is beyond the scope of the present paper.

\subsection{Numerical computation of electrostatic capacity}
\label{section:capacity}
As outlined in the introduction, the capacity of $\Omega$ can be computed by
\begin{align}\label{eq:cap}
 \capacity(\Omega) = \frac{1}{4\pi} \, \dual{\phi}{1}_\Gamma,
 \text{ where $\phi\in H^{-1/2}(\Gamma)$ solves }
 V\phi = 1 \text{ on }\Gamma=\partial\Omega.
\end{align}
This yields $\capacity(\Omega) = \frac{1}{4\pi} \, \dual{V^{-1}1}{1}_\Gamma$, and ellipticity of $V^{-1}$ guarantees $\capacity(\Omega)>0$. Let $\Phi_\ell^{\rm dual}\in \PP^0(\TT_\ell^{\rm dual})$ be the Galerkin approximation~\eqref{eq:discreteform} of $\phi=V^{-1}1$, which is obtained by Algorithm~\ref{algorithm} for $f=1$. Then, a natural approximation of the capacity~\eqref{eq:cap} is
\begin{align}\label{eq:cap:ell}
 \capacity_\ell(\Omega) := \frac{1}{4\pi} \dual{\Phi_\ell^{\rm dual}}{1}_\Gamma.
\end{align}
We note that this approximation is controlled by the energy error. 

\begin{proposition}\label{proposition:capacity}
There holds $0 < \capacity_\ell(\Omega) \le \capacity(\Omega)$ as well as
\begin{align}\label{eq:prop:cap1}
4 \pi \,(\capacity(\Omega) - \capacity_\ell(\Omega)) 
 = \enorm{\phi-\Phi_\ell^{\rm dual}}^2
 \simeq \norm{\phi-\Phi_\ell^{\rm dual}}{H^{-1/2}(\Gamma)}^2.
\end{align}
%
%
\end{proposition}

\begin{proof}
Define $\boldsymbol{1}_\ell^{\rm dual} \in \R^{M_\ell}$ by $\boldsymbol{1}_\ell^{\rm dual}[j] = \dual{1}{\chi_{\ell,j}^{\rm dual}}_\Gamma = |T_{\ell,j}^{\rm dual}| > 0$. Then, it holds that 
\begin{align*}
 \Phi_\ell^{\rm dual} = \sum_{j = 1}^{M_\ell} \mathbf{x}_\ell[j] \chi_{\ell,j}^{\rm dual},
 \quad \text{where} \quad
 \mathbf{x}_\ell := (\mathbf{V}_\ell^{\rm dual})^{-1} \boldsymbol{1}_\ell^{\rm dual}.
\end{align*}
In particular, $4\pi \, \capacity_\ell(\Omega) = \boldsymbol{1}_\ell^{\rm dual} \cdot \mathbf{x}_\ell = \boldsymbol{1}_\ell^{\rm dual} \cdot (\mathbf{V}_\ell^{\rm dual})^{-1} \boldsymbol{1}_\ell^{\rm dual}$.
Since the matrix $\mathbf{V}_\ell^{\rm dual}$ is symmetric and positive definite (and thus also its inverse $(\mathbf{V}_\ell^{\rm dual})^{-1}$), it follows that $\capacity_\ell(\Omega) > 0$.
To see~\eqref{eq:prop:cap1}, recall that the combination of~\eqref{eq:weakform}--\eqref{eq:discreteform} yields the Galerkin orthogonality $\dual{V(\phi-\Phi_\ell^{\rm dual})}{\Psi_\ell^{\rm dual}}_\Gamma = 0$ for all $\Psi_\ell^{\rm dual} \in \PP^0(\TT_\ell^{\rm dual})$. This leads to
\begin{align*}
 &4 \pi \, \big( \capacity(\Omega) - \capacity_\ell(\Omega) \big)
 = \dual{\phi - \Phi_\ell^{\rm dual}}{1}_\Gamma
 =\dual{\phi-\Phi_\ell^{\rm dual}}{V\phi}_\Gamma  
 \\& \quad
 = \dual{V(\phi-\Phi_\ell^{\rm dual})}{\phi-\Phi_\ell^{\rm dual}}_\Gamma
 = \enorm{\phi-\Phi_\ell^{\rm dual}}^2
 \ge0.
\end{align*}
This proves~\eqref{eq:prop:cap1} and, in particular, $\capacity_\ell(\Omega)\le\capacity(\Omega)$.
%
%
%
\end{proof}

\section{Numerical computations}
\label{section:computations}

In this section, we present some numerical computations that underpin our theoretical findings. 
In the experiments, we compare the performance of 
\begin{itemize}
\item Algorithm~\ref{algorithm} (with Galerkin solution $\Phi_\ell^{\rm dual} \in \PP^0(\TT_\ell^{\rm dual})$ on the dual mesh and driven by the proposed ZZ-type estimator $\eta_\ell$),
\item a standard adaptive algorithm (see Appendix~\ref{section:app:residual}) from~\cite{fkmp,fkmp:part1} with Galerkin approximation $\Phi_\ell \in \PP^0(\TT_\ell)$ on the primal mesh $\TT_\ell$ and driven by the weighted-residual error estimator, 
\end{itemize}
for different adaptivity parameters $0 < \theta \leq 1$. One particular focus is on preconditioners, where we compare
\begin{itemize}
\item Algorithm~\ref{algorithm} with its built-in operator preconditioning vs.\ diagonal preconditioning vs.\ non-preconditioning;
\item Algorithm~\ref{algorithm:res} with multilevel additive Schwarz preconditioning (MLAS)~\cite{abem+solve} vs.\ diagonal preconditioning~\cite{gm06} vs.\ non-preconditioning;
\end{itemize}
We consider lowest-order BEM for different polyhedral domains $\Omega$ starting from simple shapes, e.g., the unit cube (Section~\ref{subsection:cube}), and moving to more complicated shapes, e.g., a star-shaped domain (Section~\ref{subsection:star}). If not stated otherwise, we employ Algorithm~\ref{algorithm} with the D\"orfler parameter $\theta = 1/2$ and operator preconditioning.


We emphasize that due generic edge singularities, increasing the polynomial degree, without using anisotropic refinement, does not lead to any improved order of convergence. 
Usually one observes a slightly better constant, but the same overall algebraic rate; see, e.g.,~\cite{pointabem}.

The numerical computations were done with help
of Bempp, which is an open-source Galerkin boundary element library. 
We refer to~\cite{bempp,bempp2,bempp3} for details on Bempp. A possible implementation 
of Algorithm~\ref{algorithm} in Python with Bempp is given in Appendix~\ref{section:app:code}.

\subsection{Example (Unit cube)}
\label{subsection:cube}

We consider the unit cube $\Omega = (0,1)^3$, 
where a reference value of the capacity $\capacity(\Omega)$ is known; see Figure~\ref{fig:cube:inital:mesh}.
To our knowledge, the approximation $\capacity(\Omega) \approx 0.66067815409957$ from~\cite[Table 1]{hp13} with an error of order $10^{-13}$ is the best approximation available. 
This value is used to compute the capacity error $\err_\ell := |\capacity -\capacity_\ell|$.


\begin{figure}[h!]
	\centering
	\includegraphics[width=0.42\textwidth]{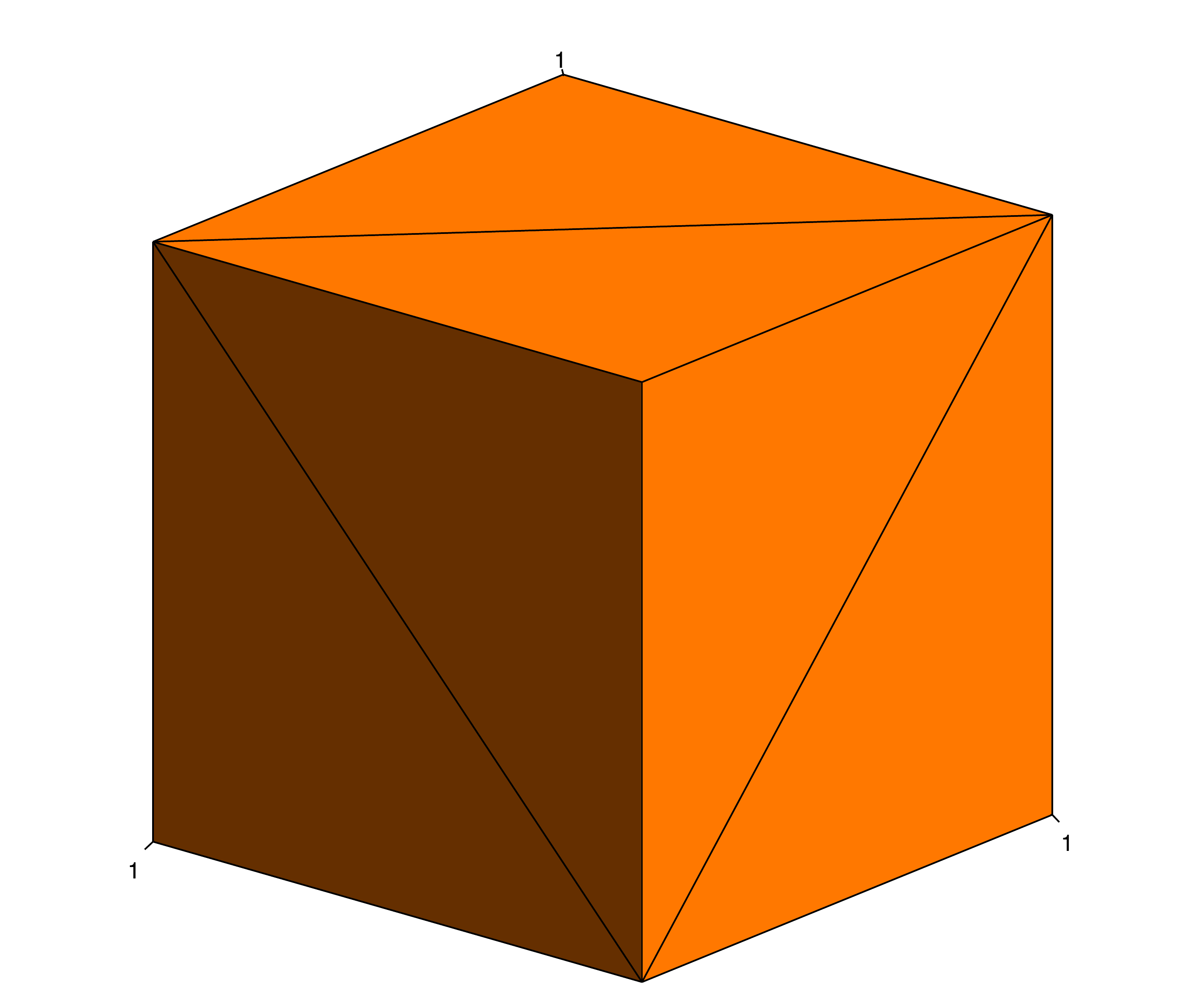} 
    \caption{Example~\ref{subsection:cube}: Initial mesh $\TT_0$ with $12$ elements. }
	\label{fig:cube:inital:mesh}   
\end{figure}

Figure~\ref{fig:cube:convergence} shows the convergence rates of the estimator $\eta_\ell^2$ and the capacity error $\err_\ell$ plotted over the number of elements $\# \TT_\ell$.
For both adaptive algorithms, the error as well as the estimator converge with rate $\OO(N^{-1})$. 
This underlines empirically that both error estimators are reliable and efficient, i.e., up to multiplicative constants equivalent to the error 
	(i.e., $\eta_\ell^2 \simeq \err = \enorm{\phi-\Phi_\ell}^2$).
Although the empirical reliability constant $\err_\ell / \eta_\ell^2$ is slightly better for Algorithm~\ref{algorithm:res} and the residual estimator,
Algorithm~\ref{algorithm} and ZZ-type estimator lead to a smaller error 
for the capacity approximation. 
Figure~\ref{fig:cube:theta} compares different values of the D\"orfler marking parameter 
$\theta \in \{0.2,0.4,0.6,0.8,1\}$ in Algorithm~\ref{algorithm}, where $ \theta = 1$ coincides with uniform
mesh-refinement.
We see that Algorithm~\ref{algorithm} is robust and any choice of $\theta <1$ leads to the optimal convergence behavior. Only uniform mesh-refinement ($\theta = 1$) leads to a sub-optimal rate $\OO(N^{-2/3})$. 
Table~\ref{table:cube:tab} displays some numerical results for the sequence of solutions produced by Algorithm~\ref{algorithm}.
It shows that, e.g., Step $13$ with $930$ elements and 12 iterations is sufficient to reach an accuracy of $5.520\cdot 10^{-4}$. 
Such a computation can be done on a usual workstation within seconds.

\bigskip

\begin{table}[h!]
\begin{tabular}{|c|c|c|c|c|c|c|}
\hline
$\ell$ & $\# \TT_\ell$ & $\capacity_\ell$ & $\eta_\ell^2$  &  $\err_\ell = |\capacity - \capacity_\ell|$ & $\err_\ell / \eta_\ell^2$ & $\#~ \text{iterations}$ \\
\hline\hline
0 &     $12$ & 	    $0.6492810516$ &    $1.702 \cdot 10^{-2}$ & $1.114 \cdot 10^{-2}$ & $0.654$ & $4$ \\
\hline
6 &     $114$ &     $0.6579150471$ &    $1.133 \cdot 10^{-1}$ & $2.763 \cdot 10^{-3}$ & $0.024$ & \revision{$12$} \\
\hline
13 &     $930$	&   $0.6601261997$ &    $2.560 \cdot 10^{-2}$ & $5.520\cdot 10^{-4}$ & $0.022$ & \revision{$12$} \\
\hline
25 &    $11132$	&    $0.6606358189$ &   $2.460 \cdot 10^{-3}$ & $4.234\cdot 10^{-5}$ & $0.017$ & \revision{$12$} \\
\hline
34 &    $96216$ &   $0.6606746264$ &    $2.473 \cdot 10^{-4}$ & $3.528 \cdot 10^{-6}$ & $0.014$ & \revision{$12$} \\
\hline
39 &    $348152$ &   $0.6606774253$ &   $5.920 \cdot 10^{-5}$ & $7.288 \cdot 10^{-7}$ & $0.012$ &\\
\hline
\end{tabular}

\caption{Example~\ref{subsection:cube}: Numerical results of Algorithm~\ref{algorithm} for $\theta = 1/2$. }
\label{table:cube:tab}
\end{table}

\begin{figure}[h!]
\begin{tikzpicture}
\begin{loglogaxis}[
 width=1.0\textwidth, height=8.5cm,
legend style={at={(0.02,0.02)},
legend cell align={left},
anchor=south west, align=left, draw = none },
xlabel={number of elements},
ylabel={error or estimator},
]
	\addplot+[solid,mark=square,mark size=2pt,mark options={line width=1.0pt},color=blue] table
		[x=number_of_elements,y=capacity_error]{figures/cube_res_200k.csv};
	\addlegendentry{\err: residual estimator}	
	
	\addplot+[solid,mark=*,mark size=2pt,mark options={line width=1.0pt},color=blue] table
		[x=number_of_elements,y=estimator]{figures/cube_res_200k.csv};
	\addlegendentry{$\eta_\ell^2$: residual-estimator}
	
		\addplot+[solid,mark=square,mark size=2pt,mark options={line width=1.0pt},color=red] table
		[x=number_of_elements,y=capacity_error]{figures/cube_zz_200k.csv};
	\addlegendentry{\err: ZZ-estimator}	
	
	\addplot+[solid,mark=*,mark size=2pt,mark options={line width=1.0pt},color=red] table
		[x=number_of_elements,y=estimator]{figures/cube_zz_200k.csv};
	\addlegendentry{$\eta_\ell^2$: ZZ-estimator}


	\addplot [black,dashed ] expression [domain=24:400000
		, samples = 10] {40*x^(-1)} node [midway,above,yshift=0.25cm] {$\mathcal{O}(N^{-1})$};

\end{loglogaxis}
\end{tikzpicture}
\caption{Example~\ref{subsection:cube}: Convergence of the error estimator $\eta_\ell^2$ and capacity error $\err_\ell$. 
The estimator $\eta_\ell^2$ as well as the error converge with optimal order 
$\OO(N^{-1})$, if the a weighted-residual or the proposed ZZ-based error estimator is used.}
\label{fig:cube:convergence}  
\end{figure}
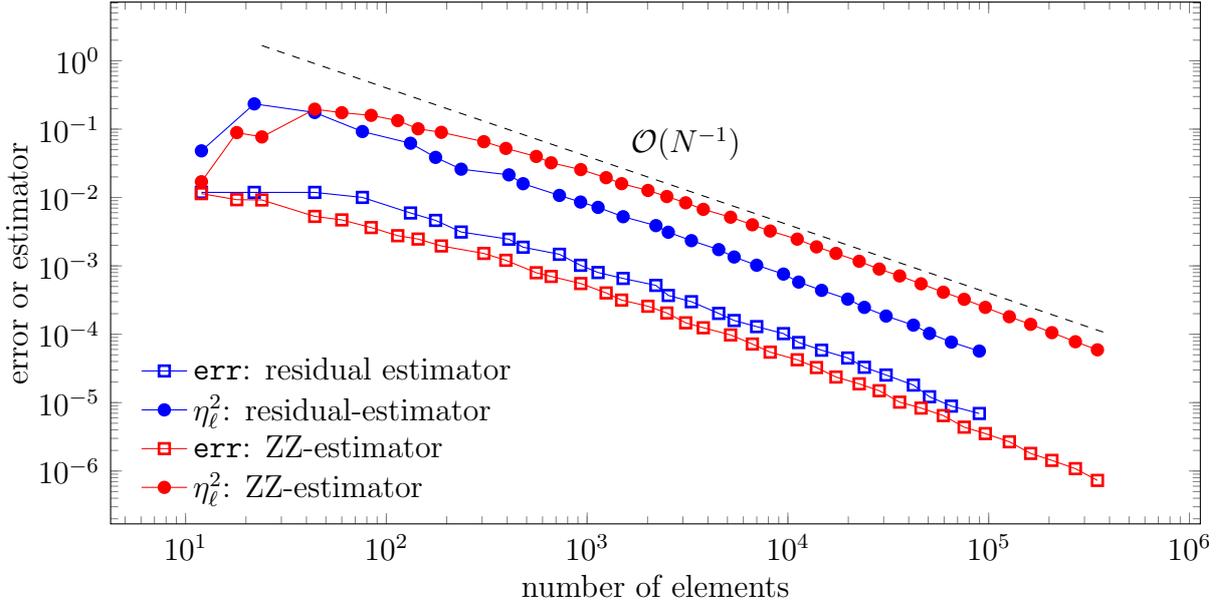

\begin{figure}[h!]
\begin{tikzpicture}
\begin{loglogaxis}[
 width=1.0\textwidth, height=8.5cm,
legend style={at={(0.02,0.02)},
legend cell align={left},
anchor=south west, align=left, draw = none },
xlabel={number of elements},
ylabel={error or estimator},
]

	\addplot+[solid,mark=square,mark size=2pt,mark options={line width=1.0pt},color=red] table
		[x=number_of_elements,y=capacity_error]{figures/cube_theta_02.csv};
	\addlegendentry{$\theta = 0.2$}	

	\addplot+[solid,mark=square,mark size=2pt,mark options={line width=1.0pt},color=orange] table
		[x=number_of_elements,y=capacity_error]{figures/cube_theta_04.csv};
	\addlegendentry{$\theta = 0.4$}		

	\addplot+[solid,mark=square,mark size=2pt,mark options={line width=1.0pt},color=blue] table
		[x=number_of_elements,y=capacity_error]{figures/cube_theta_06.csv};
	\addlegendentry{$\theta = 0.6$}	

	\addplot+[solid,mark=square,mark size=2pt,mark options={line width=1.0pt},color=green] table
		[x=number_of_elements,y=capacity_error]{figures/cube_theta_08.csv};
	\addlegendentry{ $\theta = 0.8$}	
	
    \addplot+[solid,mark=square,mark size=2pt,mark options={line width=1.0pt},color=magenta] table
		[x=number_of_elements,y=capacity_error]{figures/cube_theta_1.csv};
	\addlegendentry{uniform}	

	\addplot+[solid,mark=*,mark size=2pt,mark options={line width=1.0pt},color=red] table
		[x=number_of_elements,y=estimator]{figures/cube_theta_02.csv};
	
	\addplot+[solid,mark=*,mark size=2pt,mark options={line width=1.0pt},color=orange] table
		[x=number_of_elements,y=estimator]{figures/cube_theta_04.csv};

	\addplot+[solid,mark=*,mark size=2pt,mark options={line width=1.0pt},color=blue] table
		[x=number_of_elements,y=estimator]{figures/cube_theta_06.csv};
	
	\addplot+[solid,mark=*,mark size=2pt,mark options={line width=1.0pt},color=green] table
		[x=number_of_elements,y=estimator]{figures/cube_theta_08.csv};
	
	\addplot+[solid,mark=*,mark size=2pt,mark options={line width=1.0pt},color=magenta] table
		[x=number_of_elements,y=estimator]{figures/cube_theta_1.csv};

	\addplot [black,dashed ] expression [domain=24:100000
	, samples = 10] {1.5* 10^(-1)*x^(-2/3)} node [midway,above,yshift=0.10cm] {$\mathcal{O}(N^{-2/3})$};

	\addplot [black,dashed ] expression [domain=24:100000
		, samples = 10] {40*x^(-1)} node [midway,above,yshift=0.25cm] {$\mathcal{O}(N^{-1})$};

\end{loglogaxis}
\end{tikzpicture}
\caption{Example~\ref{subsection:cube}: Convergence of the error estimator $\eta_\ell^2$ (circles) and capacity error $\err_\ell$ (squares) for different value of $0 < \theta <1$ and uniform refinement $\theta = 1$.
Uniform refinement leads to a reduced rate of convergence. 
}
\label{fig:cube:theta}  
\end{figure}
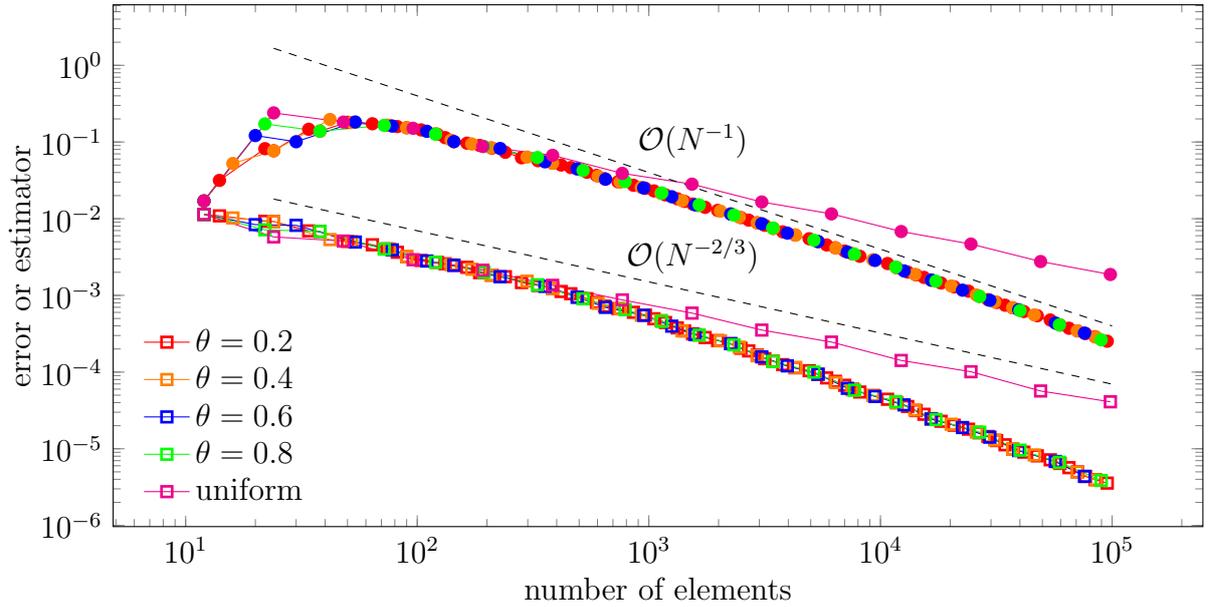

Figure~\ref{fig:cube:iterations} displays the condition number of the arising linear systems~\eqref{eq:discreteform}.
We emphasize that the MLAS preconditioning is 
optimal in the sense that the condition number of the arising systems remains bounded, independently of the number of elements; see~\cite{abem+solve}. 
This is confirmed for MLAS as well as empirically observed also for operator preconditioning with the hypersingular operator in Algorithm~\ref{algorithm}. On the other hand, the condition number of the non-preconditioned systems explodes with increasing number of elements resp.\ mesh graduation. 
Moreover, Figure~\ref{fig:cube:iterations} shows that diagonal preconditioning does lead to a slower growth of the condition number, but is not optimal. 
Figure~\ref{fig:cube:res10_15} and Figure~\ref{fig:cube:res20_30} show the distribution of the error estimator in some of the generated adaptive meshes. Due to generic edge singularities, the adaptive algorithm leads to refinement along edges and corners, while elements in the interior of the cube-facets stay relatively coarse. 

\begin{figure}[h!]
\begin{tikzpicture}
\begin{loglogaxis}[
 width=1.0\textwidth, height=8.5cm,
legend style={at={(0.02,1-0.02)},
legend cell align={left}, anchor=north west, align=left, draw = none },
xlabel={number of elements},
ylabel={condition number},
]
	\addplot+[solid,mark=*,mark size=2pt,mark options={line width=1.0pt},color=red] table
		[x=number_of_elements,y=cond_precon]{figures/cube_cond_zz.csv};
	\addlegendentry{ZZ - precon}
	
	\addplot+[solid,mark=*,mark size=2pt,mark options={line width=1.0pt},color=blue] table
		[x=number_of_elements,y=cond_diag]{figures/cube_cond_zz.csv};
	\addlegendentry{ZZ - diag}	
	
	\addplot+[solid,mark=*,mark size=2pt,mark options={line width=1.0pt},color=green] table
		[x=number_of_elements,y=cond_non]{figures/cube_cond_zz.csv};
	\addlegendentry{ZZ - non}	
	
	\addplot+[solid,mark=square,mark size=2pt,mark options={line width=1.0pt},color=red] table
		[x=number_of_elements,y=cond_precon]{figures/cube_cond_res.csv};
	\addlegendentry{Res - precon}
	
	\addplot+[solid,mark=square,mark size=2pt,mark options={line width=1.0pt},color=blue] table
		[x=number_of_elements,y=cond_diag]{figures/cube_cond_res.csv};
	\addlegendentry{Res - diag}	
	
	\addplot+[solid,mark=square,mark size=2pt,mark options={line width=1.0pt},color=green] table
		[x=number_of_elements,y=cond_non]{figures/cube_cond_res.csv};
	\addlegendentry{Res - non}


\end{loglogaxis}
\end{tikzpicture}
\caption{Example~\ref{subsection:cube}: Condition number for different preconditioning strategies for the arising linear systems in Algorithm~\ref{algorithm} (ZZ) with operator preconditioning with the hypersingular operators vs. Algorithm (Res) with multilevel additive Schwarz preconditioning. Additionally, diagonal preconditioning is employed in both cases.  }
\label{fig:cube:iterations}  
\end{figure}
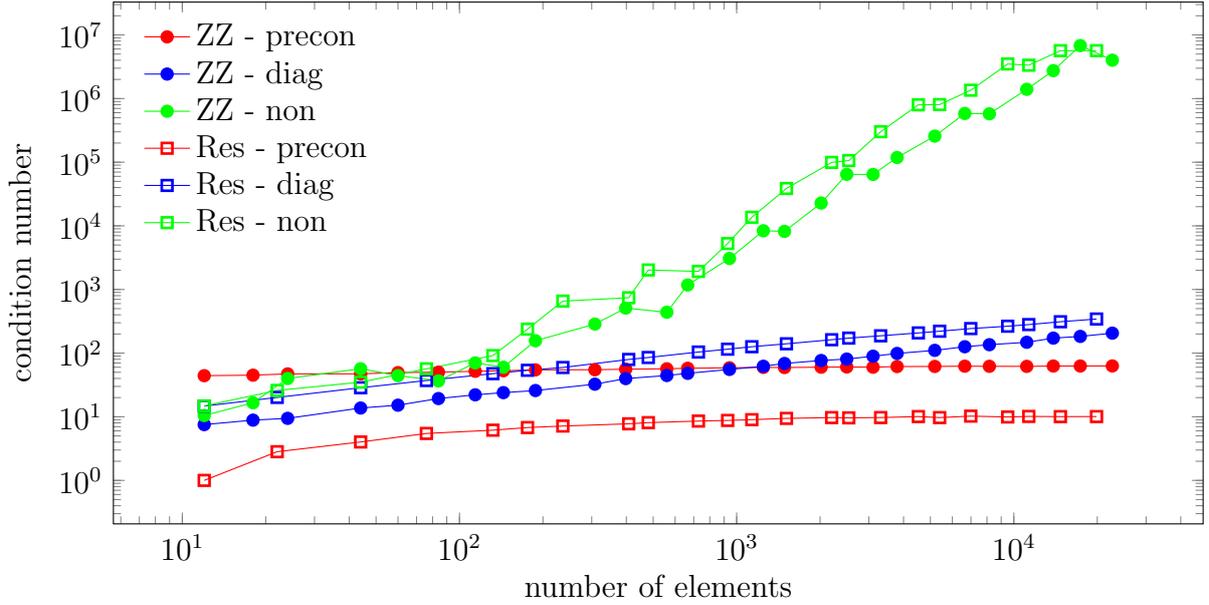

\begin{figure}
	\centering
	\includegraphics[width=0.48\textwidth]{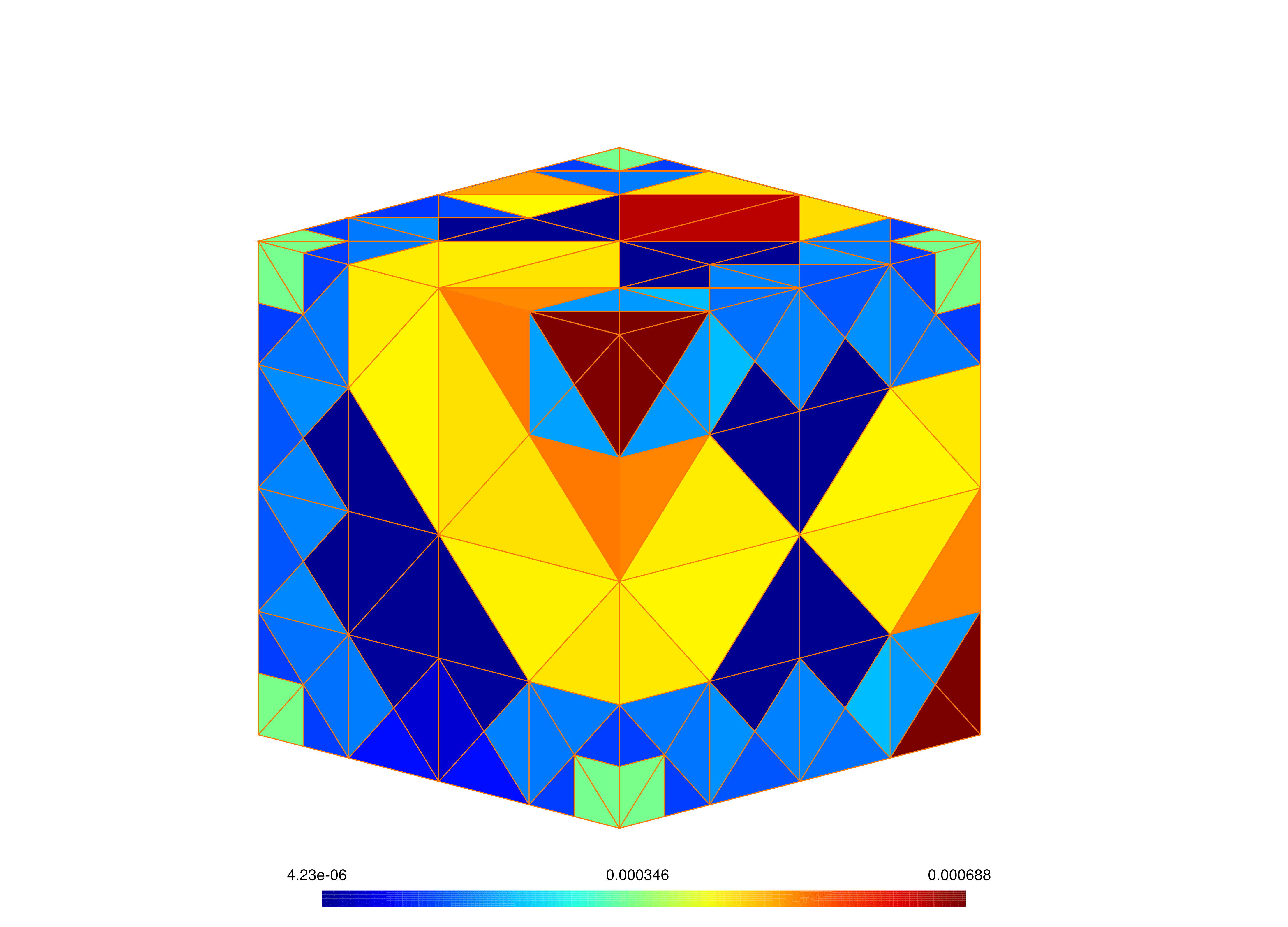}
	\includegraphics[width=0.48\textwidth]{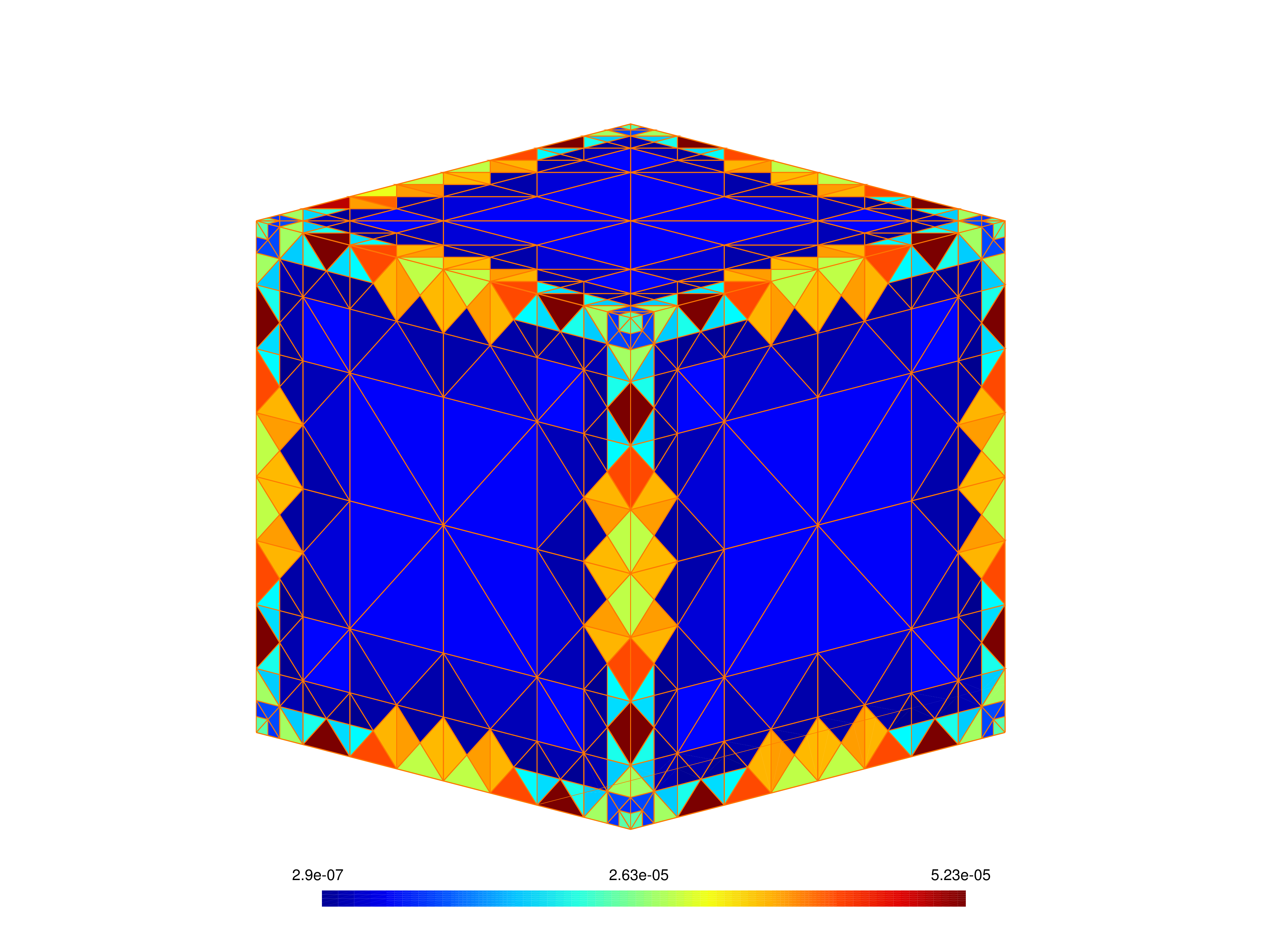}
	\caption{Example~\ref{subsection:cube}: Mesh $\TT_\ell$ and distribution of the error estimator $\eta_\ell^2$ for $\ell = 10$ with $306$ elements (left) and for $\ell = 15$  with $1244$ elements (right).}
	\label{fig:cube:res10_15}	  
\end{figure}

\begin{figure}
	\centering
	\includegraphics[width=0.48\textwidth]{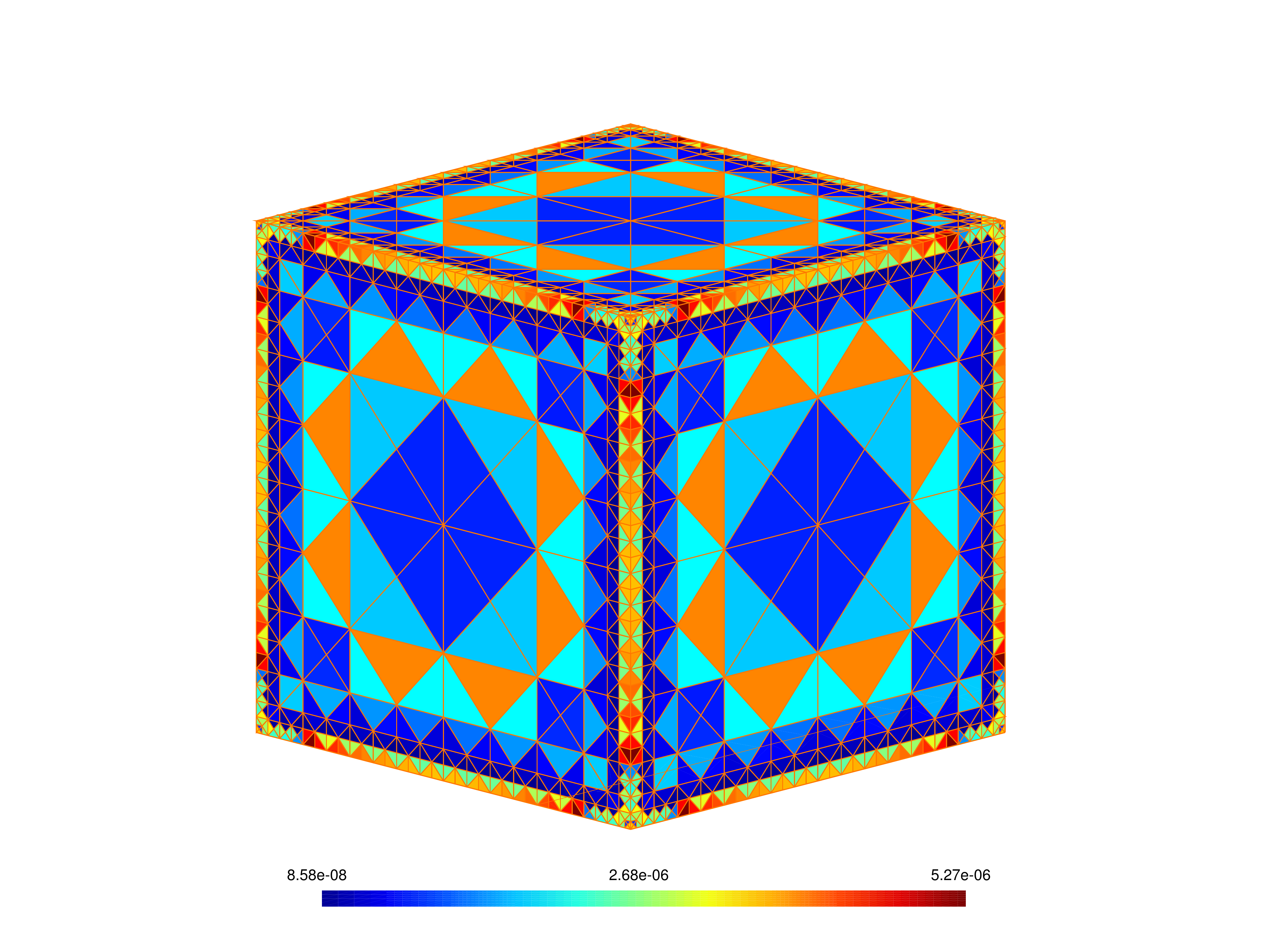}
	\includegraphics[width=0.48\textwidth]{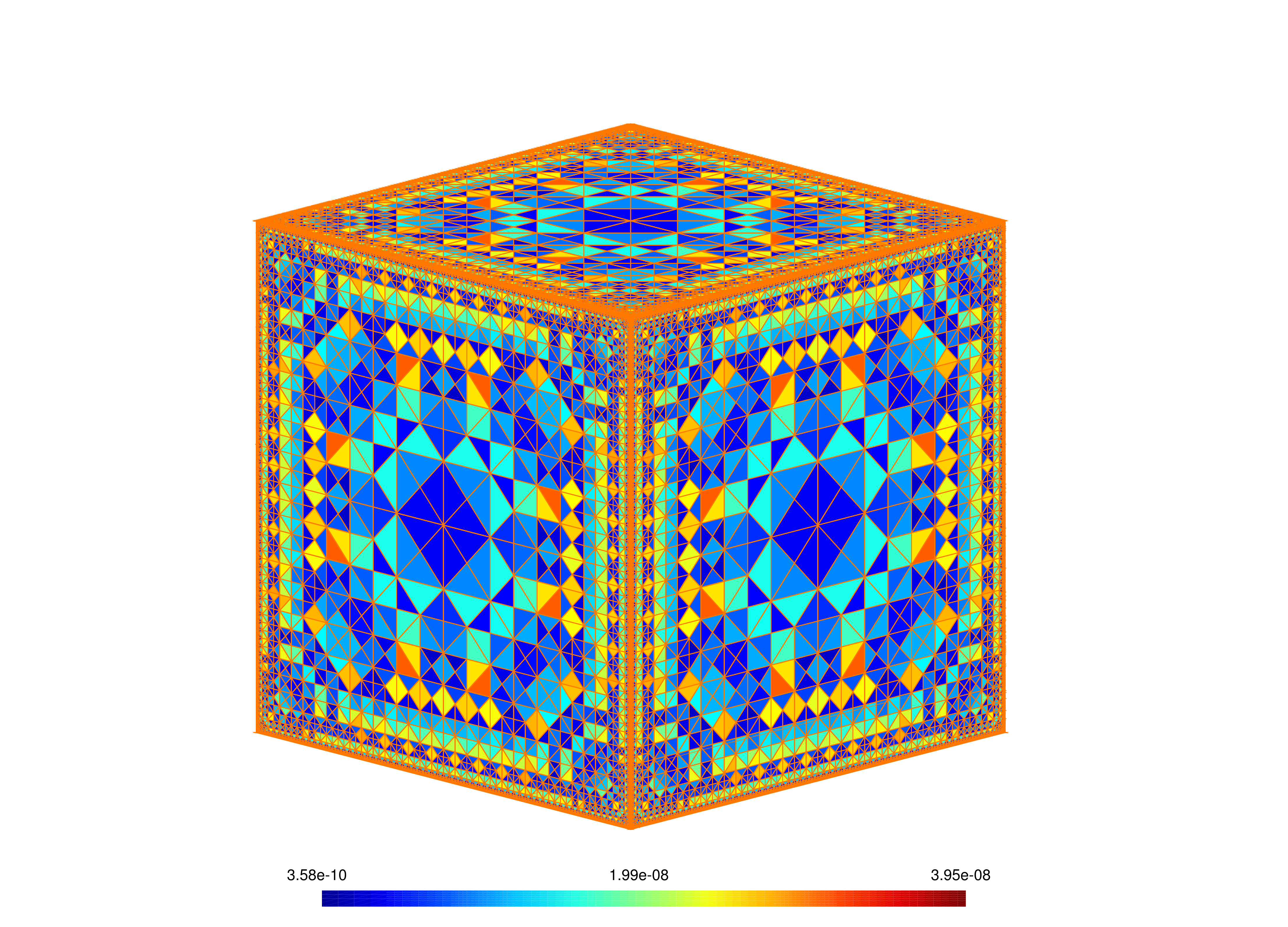}
	\caption{Example~\ref{subsection:cube}: Mesh $\TT_\ell$ and distribution of the error estimator $\eta_\ell^2$ for $\ell =20$ with $3792$ elements (left) and for $\ell =30$ with $46044$ elements (right).}
	\label{fig:cube:res20_30}	  
\end{figure}


\subsection{Example 2: Fichera cube}
\label{subsection:fichera}

As second example, we consider the Fichera cube $[-1,1]^3 \setminus [0,1]^3$  with side length $2$; see Figure~\ref{fig:fichera:inital} (left). 
In contrast to the unit cube, we lack a highly accurate benchmark computation and approximate the capacity error by
\begin{align*}
    \err_\ell := |\capacity - \capacity_\ell| \approx |\capacity_{\ell_{\rm max}} - \capacity_\ell| \quad \text{for all } 0 \leq \ell < \ell_{\rm max}.
\end{align*}
Convergence behaviour of the estimator $\eta_\ell^2$ and of the error $\err_\ell$ are displayed in 
Figure~\ref{fig:fichera:convergence}. Both, Algorithm~\ref{algorithm} as well as Algorithm~\ref{algorithm:res}, lead to the optimal rate of convergence $\OO(N^{-1})$. 
Table~\ref{table:fichera:tab} shows some numerical results produced by Algorithm~\ref{algorithm}.
For the fichera cube approximately $10^4$ elements are sufficient
to achieve an error of order $10^{-4}$.
This can be computed easily on a workstation within a couple of minutes. 
Figure~\ref{fig:fichera:inital} and Figure~\ref{fig:fichera:res} show the distribution of the error estimator in some of the generated adaptive meshes. Due to generic edge singularities, the adaptive algorithm leads to refinement along edges and corners, while elements in the interior of the cube faces stay relatively coarse.

\begin{table}[h!]
\begin{tabular}{|c|c|c|c|c|c|}
\hline
step $\ell$ & number of elements $\# \TT_\ell$ & $\capacity_\ell$ & $\eta_\ell^2$ &  $|\capacity_{\ell_{\rm max}} - \capacity_\ell|$ \\
\hline\hline
0 & $48$ & 	$1.2813985240$ & $3.239 \cdot 10^{-1}$	&	$9.858 \cdot 10^{-3}$ \\
\hline
4 & $126$ & $1.2853002501$ & $2.776 \cdot 10^{-1}$	& 	$5.957 \cdot 10^{-3}$ \\
\hline
9 & $986$ & $1.2900557132$ & $5.470 \cdot 10^{-2}$		&	$1.201 \cdot 10^{-3}$ \\
\hline
19 & $11992$ & $1.2911689924$ & $5.270\cdot 10^{-3}$ & 	$8.776\cdot 10^{-5}$ \\
\hline
28 & $113484$ & $1.2912534661$ & $4.829 \cdot 10^{-4}$ &		$3.282\cdot 10^{-6}$ \\
\hline
30 & $183762$ & $1.2912567475$ &  $2.868 \cdot 10^{-4}$ &\\
\hline
\end{tabular}
\caption{Example~\ref{subsection:fichera}: Numerical results of Algorithm~\ref{algorithm} for $\theta = 1/2$. }
\label{table:fichera:tab}
\end{table}

\begin{figure}[h!]
\begin{tikzpicture}
\begin{loglogaxis}[
 width=1.0\textwidth, height=8.5cm,
legend style={at={(0.02,0.02)},
legend cell align={left}, anchor=south west, align=left, draw = none },
xlabel={number of elements},
ylabel={error or estimator},
]
	\addplot+[solid,mark=square,mark size=2pt,mark options={line width=1.0pt},color=blue] table
		[x=number_of_elements,y=capacity_error]{figures/fichera_res_200k_cap_error.csv};
	\addlegendentry{\err: residual estimator}	
	
	\addplot+[solid,mark=*,mark size=2pt,mark options={line width=1.0pt},color=blue] table
		[x=number_of_elements,y=estimator]{figures/fichera_res_200k.csv};
	\addlegendentry{$\eta_\ell^2$: residual-estimator}
	
		\addplot+[solid,mark=square,mark size=2pt,mark options={line width=1.0pt},color=red] table
		[x=number_of_elements,y=capacity_error]{figures/fichera_zz_200k_cap_error.csv};
	\addlegendentry{\err: ZZ-estimator}	
	
	\addplot+[solid,mark=*,mark size=2pt,mark options={line width=1.0pt},color=red] table
		[x=number_of_elements,y=estimator]{figures/fichera_zz_200k.csv};
	\addlegendentry{$\eta_\ell^2$: ZZ-estimator}


	\addplot [black,dashed ] expression [domain=48:200000
		, samples = 10] {100*x^(-1)} node [midway,above,yshift=0.25cm] {$\mathcal{O}(N^{-1})$};

\end{loglogaxis}
\end{tikzpicture}
\caption{Example~\ref{subsection:fichera}: Convergence of the error estimator $\eta_\ell^2$ and capacity error $\err_\ell$. 
	The estimator $\eta_\ell^2$ as well as the error converge with optimal order 
	$\OO(N^{-1})$, if the weighted-residual or the proposed ZZ-based error estimator is used.}
\label{fig:fichera:convergence}  
\end{figure}
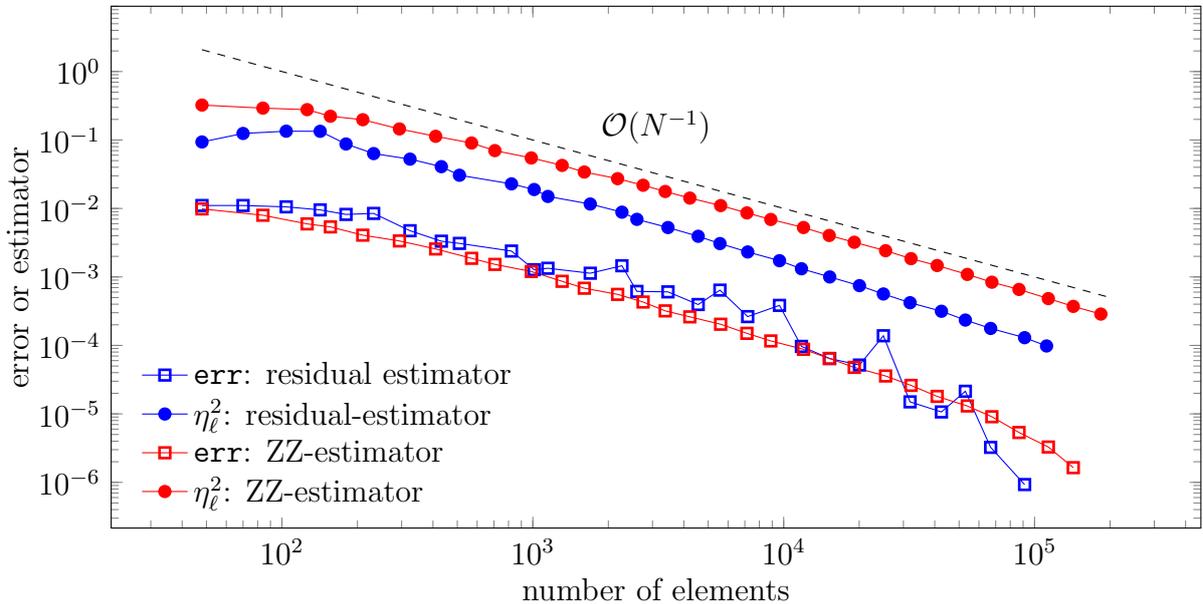

\begin{figure}
	\centering
	\includegraphics[width=0.48\textwidth]{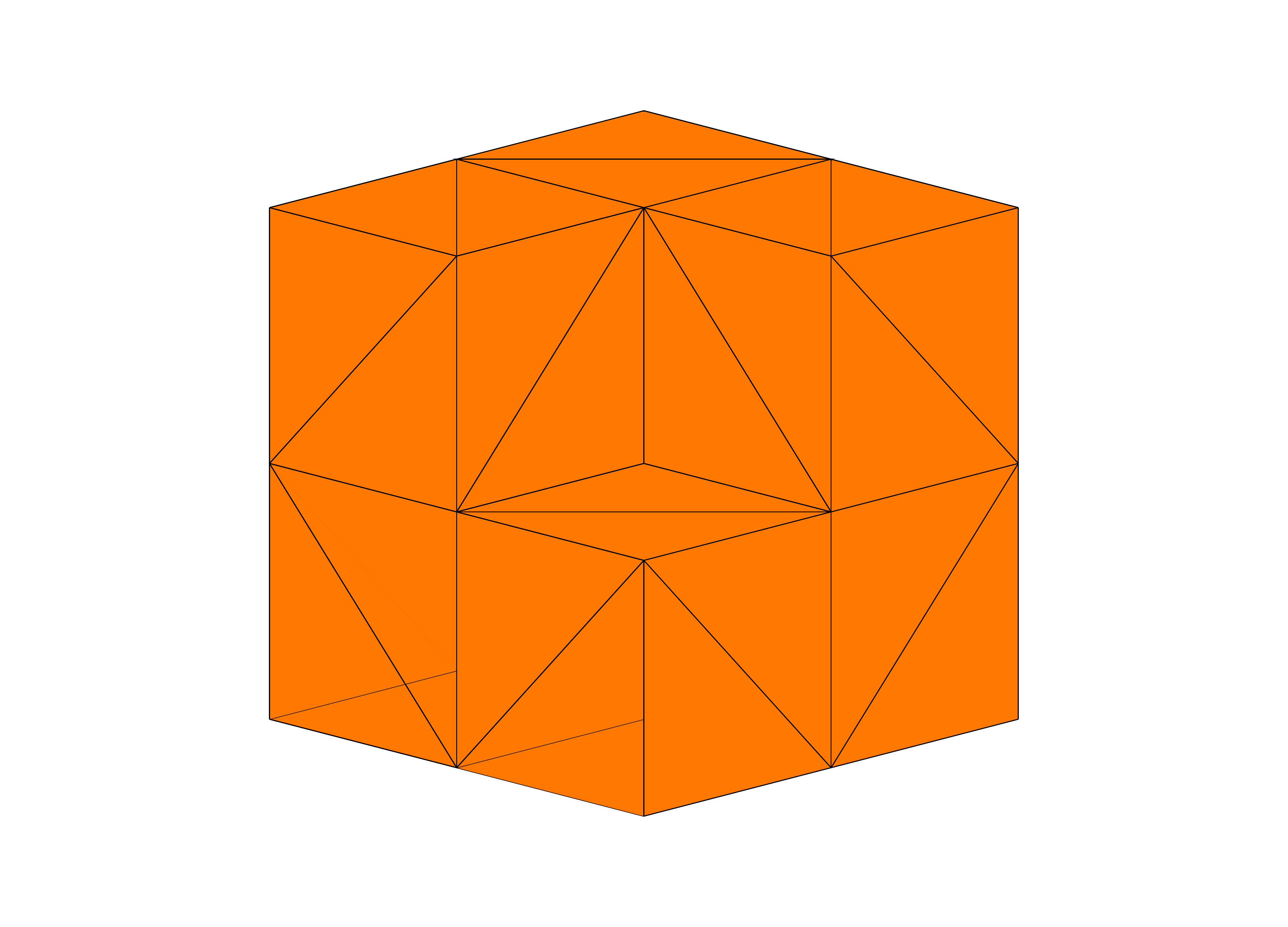}
	\includegraphics[width=0.48\textwidth]{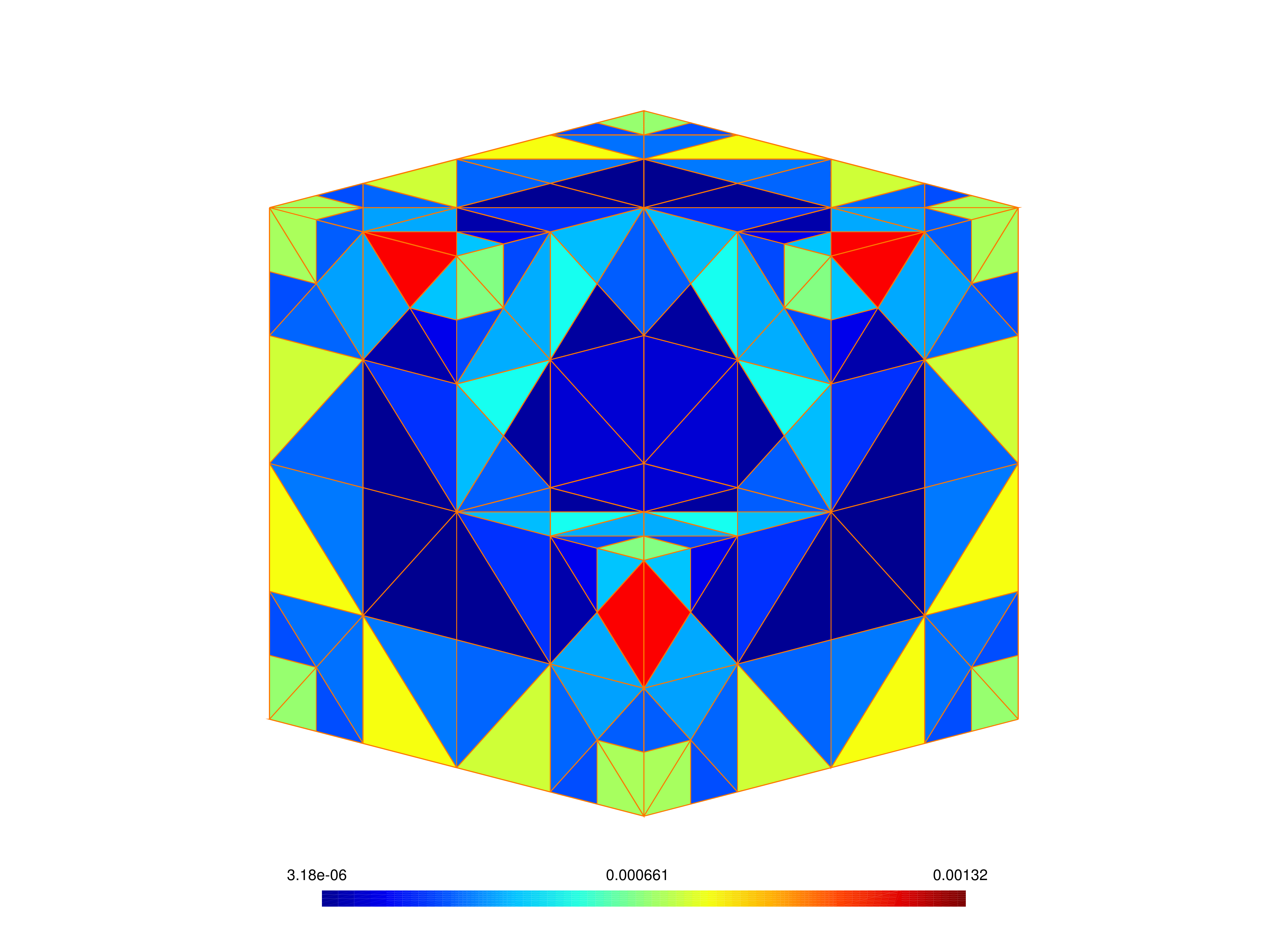}
    \caption{Example~\ref{subsection:fichera}: Initial triangulation $\TT_0$ (left) consisting of $48$ triangles.
    	Mesh $\TT_\ell$ and distribution of the error estimator $\eta_\ell^2$ for $\ell=5$ with $410$ elements (right).}
	\label{fig:fichera:inital}		  
\end{figure}

\begin{figure}
	\centering
	\includegraphics[width=0.48\textwidth]{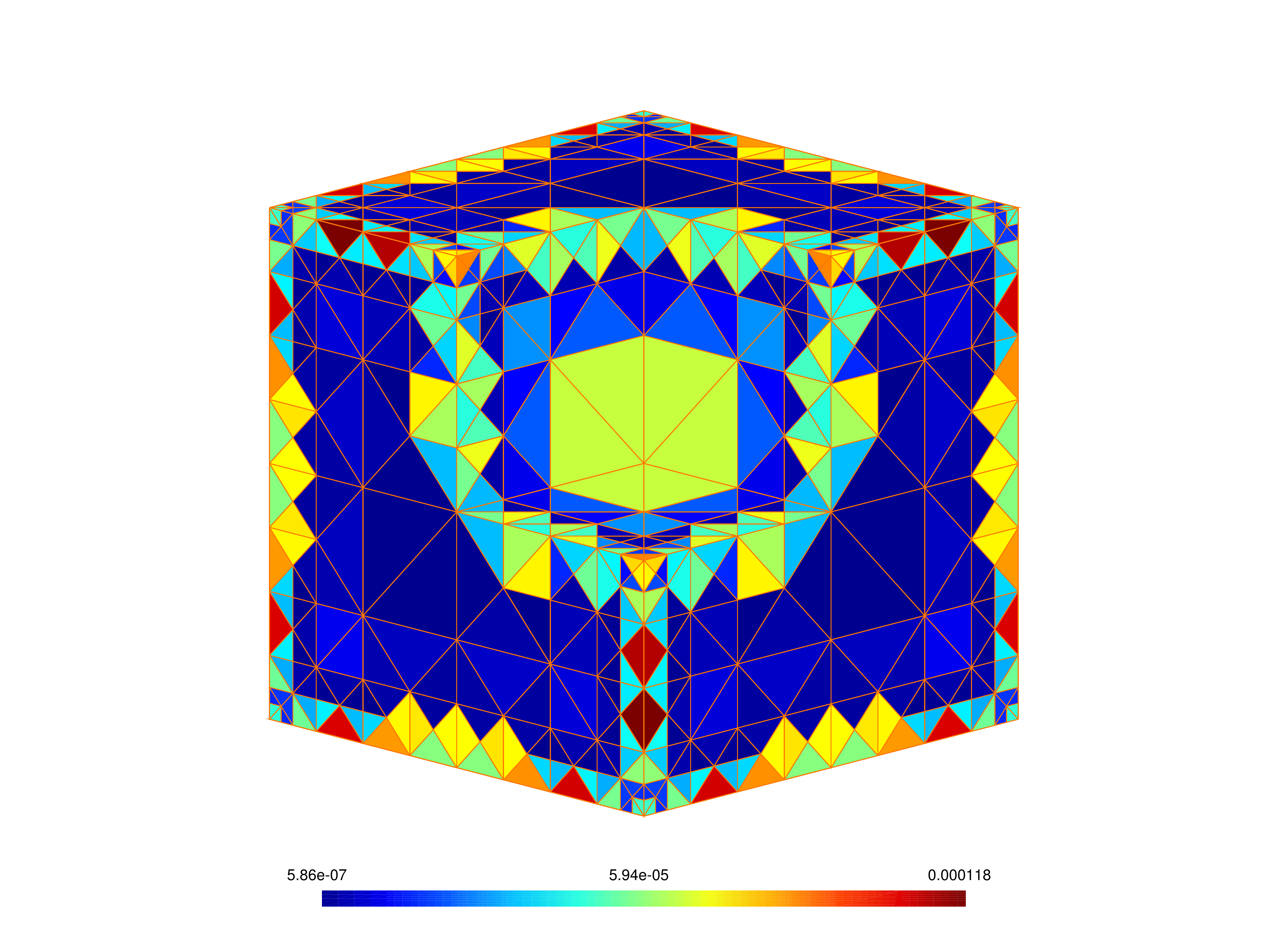}
	\includegraphics[width=0.48\textwidth]{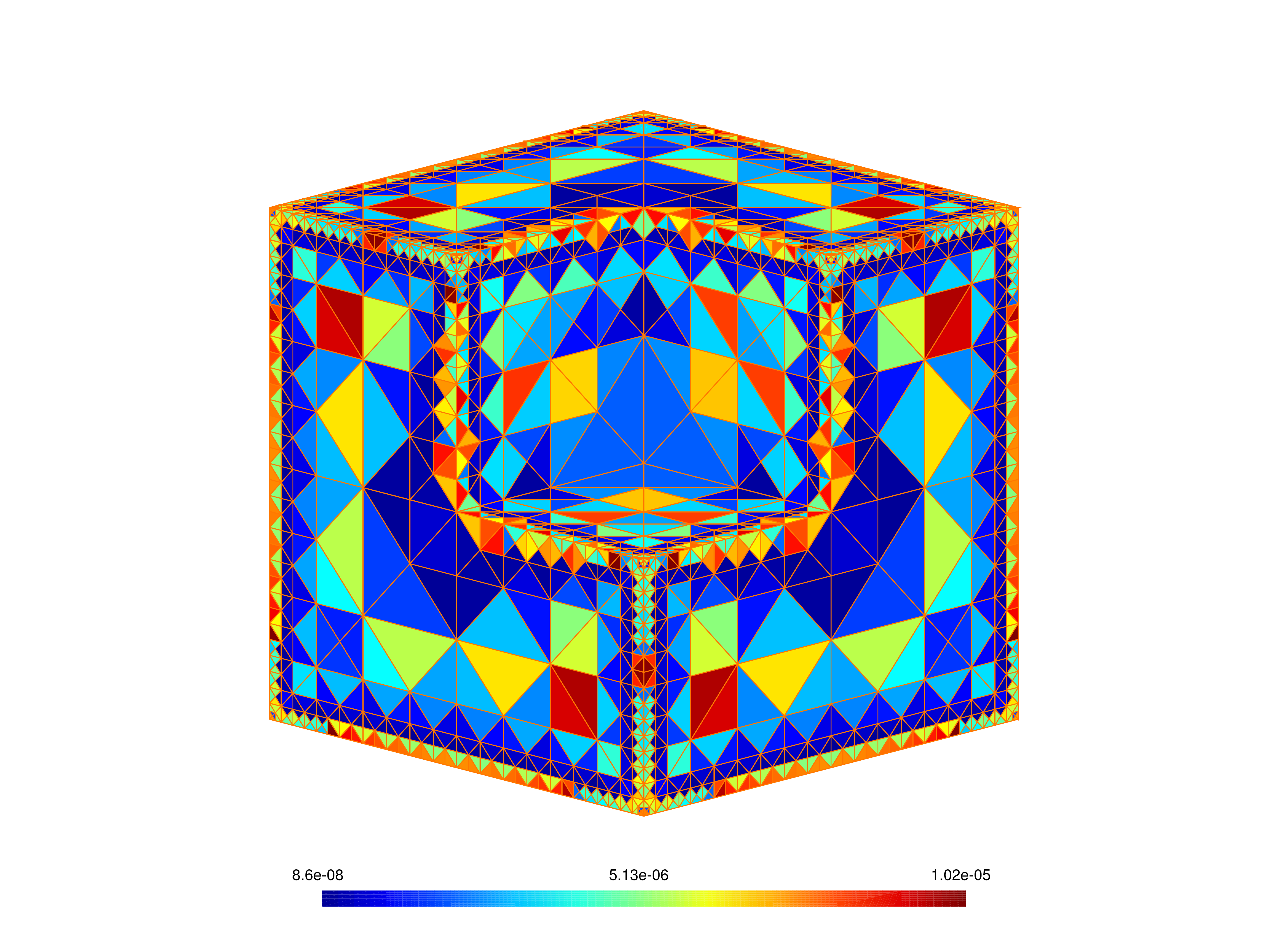}
    \caption{Example~\ref{subsection:fichera}: Mesh $\TT_\ell$ and distribution of the error estimator $\eta_\ell^2$ for $\ell=10$ (left) with $1308$ elements and $\ell = 15$ (right) with $4226$ elements.}	
	\label{fig:fichera:res}  
\end{figure}

\newpage

\subsection{Example 3: Star-shaped geometry}
\label{subsection:star}

As third example, we consider a star-shaped geometry shown in Figure~\ref{fig:star:inital}.
We use Algorithm~\ref{algorithm} as black-box method. 
Figure~\ref{fig:star:convergence} shows the convergence of the estimator $\eta_\ell^2$ and the capacity error 
$\err_\ell := |\capacity_{\ell_{\rm max}}-\capacity_\ell| \approx |\capacity-\capacity_\ell|$. 
We observe the optimal convergence order $\OO(N^{-1})$ for the estimator as well as the error for Algorithm~\ref{algorithm} and Algorithm~\ref{algorithm:res}.
The rate of convergence is robust with respect to the D\"orfler marking parameter $0 < \theta \leq 1$.  Besides
uniform refinement ($\theta = 1$), all tested parameters $\theta \in \{0.2,0.4,0.6,0.8\}$
lead to the same rate of $\OO(N^{-1})$; see Figure~\ref{fig:star:theta}.

Figure~\ref{fig:star:iterations} shows the condition number of the arising linear systems~\eqref{eq:discreteform}.
Operator preconditioning with the hypersingular operator in Algorithm~\ref{algorithm} as well as the multilevel additive Schwarz preconditioner guarantee a bounded condition number. 
On the other hand, diagonal- and non-preconditioning lead to an increasing condition number as the number of elements grows.  
Similarly to previous examples, the Algorithm has to resolve generic edge singularities.  
Figure~\ref{fig:star:res5_10} and Figure~\ref{fig:star:res20_25} show the 
distribution of $\eta_\ell^2$ on some meshes generated by Algorithm~\ref{algorithm}.

Table~\ref{table:star:tab} shows some numerical results for the star-example. 
Algorithm~\ref{algorithm} takes 22 steps and approximately $10^4$ elements 
to reach an error of order $10^{-4}$.

\begin{figure}
	\centering
	\includegraphics[width=0.48\textwidth]{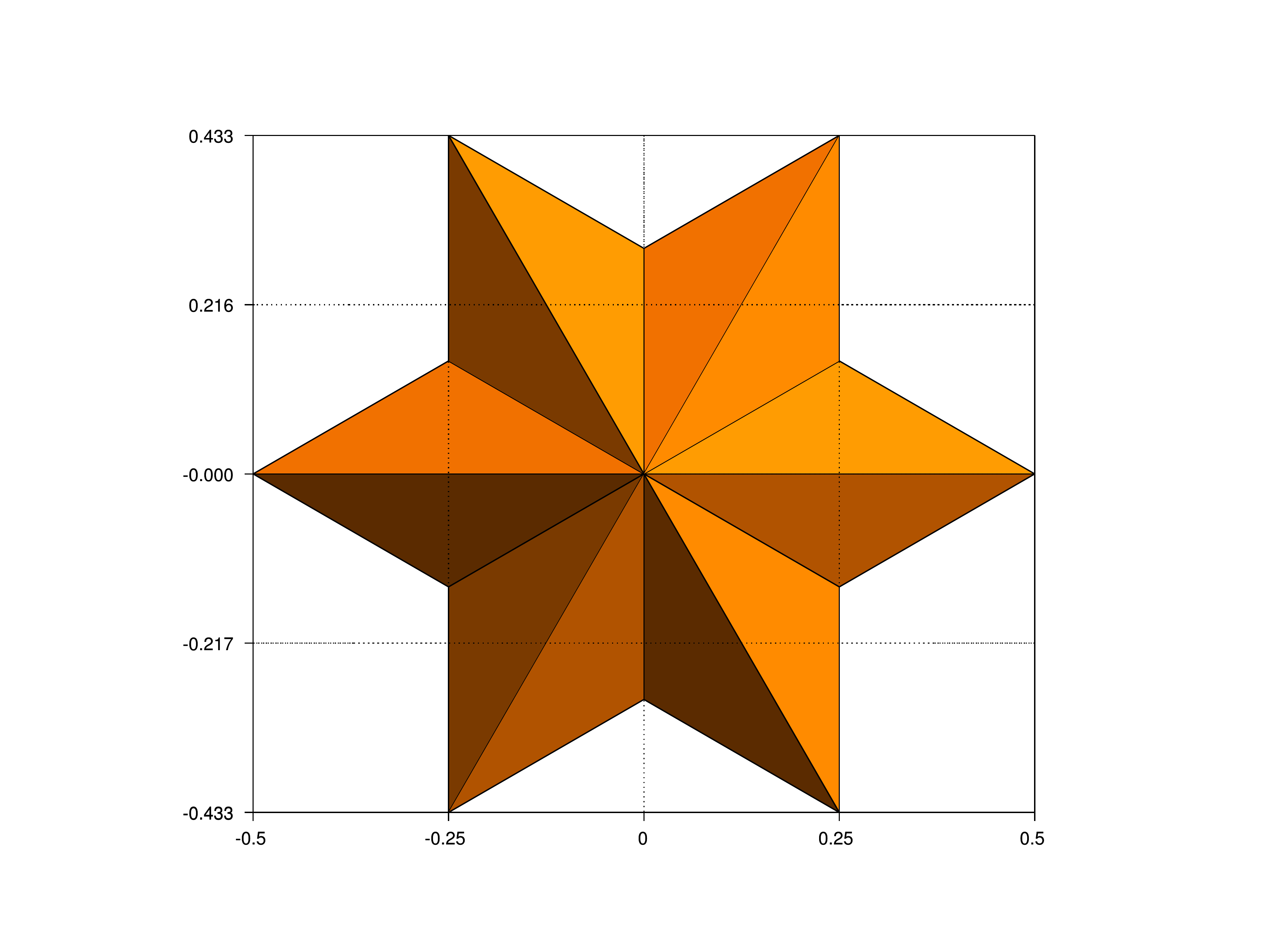}
	\includegraphics[width=0.48\textwidth]{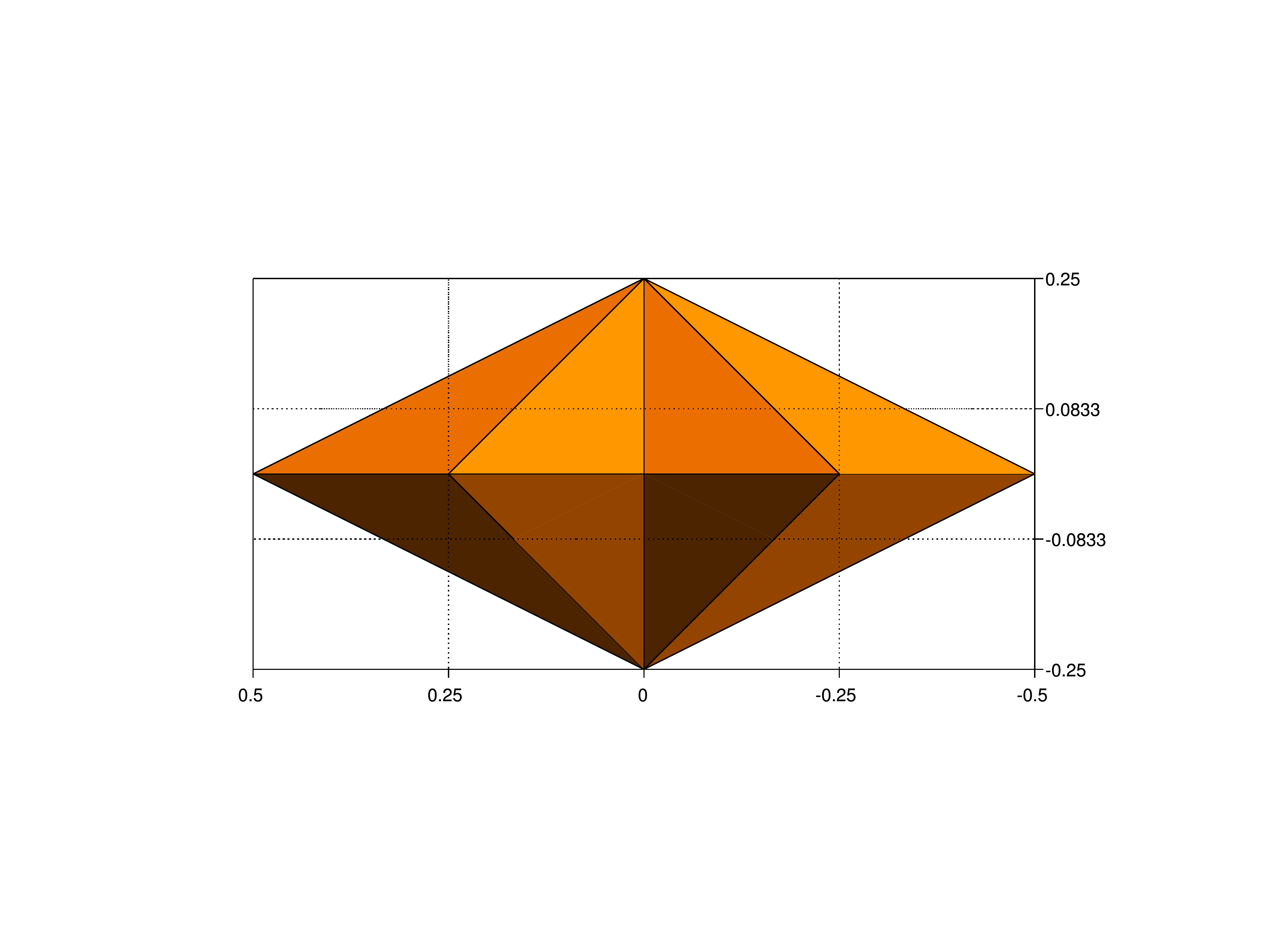}
	\caption{Example~\ref{subsection:star}: Geometry and initial mesh $\TT_0$ with $24$ elements of the star example. 
    The point of view is along the $z$-axis (left) and along the $y$-axis (right).}
	\label{fig:star:inital}	  
\end{figure}

\begin{table}[h!]
\begin{tabular}{|c|c|c|c|c|c|}
\hline
step $\ell$ & number of elements $\# \TT_\ell$ & $\capacity_\ell$ & $\eta_\ell^2$ &  $|\capacity_{\ell_{\rm max}} - \capacity_\ell|$ &
\revision{$\#~ \text{iterations}$} \\
\hline\hline
0 & $24$ & 	$0.2942967843$ & $2.580 \cdot 10^{-1}$	&	$5.456 \cdot 10^{-3}$ & \revision{$5$}\\
\hline
5 & $100$ & 	$0.2978622929$ & $7.092 \cdot 10^{-2}$	& 	$5.456 \cdot 10^{-3}$ & \revision{$5$}\\
\hline
14 & $1136$	& $0.2995412185$ & $1.384 \cdot 10^{-2}$		&	$1.686\cdot 10^{-3}$ & $\revision{20}$ \\
\hline
22 & $9752$	& $0.2998567063$ & $2.402\cdot 10^{-3}$ & 	$3.223\cdot 10^{-4}$ &$\revision{20}$ \\
\hline
31 & $91046$ & $0.2998234547$ & $2.540 \cdot 10^{-4}$ &		$2.869 \cdot 10^{-5}$ &$\revision{22}$\\
\hline
34 & $194838$ & $0.2998587081$ &  $1.110 \cdot 10^{-4}$  &&\\
\hline
\end{tabular}
\caption{Example~\ref{subsection:star}: Numerical results of Algorithm~\ref{algorithm} with $\theta = 1/2$.} 

\label{table:star:tab}
\end{table}

\begin{figure}[h!]
\begin{tikzpicture}
\begin{loglogaxis}[
 width=1.0\textwidth, height=8.5cm,
legend style={at={(0.02,0.02)},
legend cell align={left}, anchor=south west, align=left, draw = none },
xlabel={number of elements},
ylabel={error or estimator},
]
	\addplot+[solid,mark=square,mark size=2pt,mark options={line width=1.0pt},color=blue] table
		[x=number_of_elements,y=capacity_error]{figures/star_res_200k_cap_error.csv};
	\addlegendentry{\err: residual estimator}	
	
	\addplot+[solid,mark=*,mark size=2pt,mark options={line width=1.0pt},color=blue] table
		[x=number_of_elements,y=estimator]{figures/star_res_200k.csv};
	\addlegendentry{$\eta_\ell^2$: residual-estimator}
	
		\addplot+[solid,mark=square,mark size=2pt,mark options={line width=1.0pt},color=red] table
		[x=number_of_elements,y=capacity_error]{figures/star_zz_200k_cap_error.csv};
	\addlegendentry{\err: ZZ-estimator}	
	
	\addplot+[solid,mark=*,mark size=2pt,mark options={line width=1.0pt},color=red] table
		[x=number_of_elements,y=estimator]{figures/star_zz_200k.csv};
	\addlegendentry{$\eta_\ell^2$: ZZ-estimator}


	\addplot [black,dashed ] expression [domain=24:200000
		, samples = 10] {4*10*x^(-1)} node [midway,above,yshift=0.25cm] {$\mathcal{O}(N^{-1})$};

\end{loglogaxis}
\end{tikzpicture}
\caption{Example~\ref{subsection:star}: Convergence of the error estimator $\eta_\ell^2$ and capacity error $\err_\ell$. 
The estimator $\eta_\ell^2$ as well as the error converge with optimal order 
$\OO(N^{-1})$, if a weighted-residual or the proposed ZZ-based error estimator is used.}
\label{fig:star:convergence}  
\end{figure}
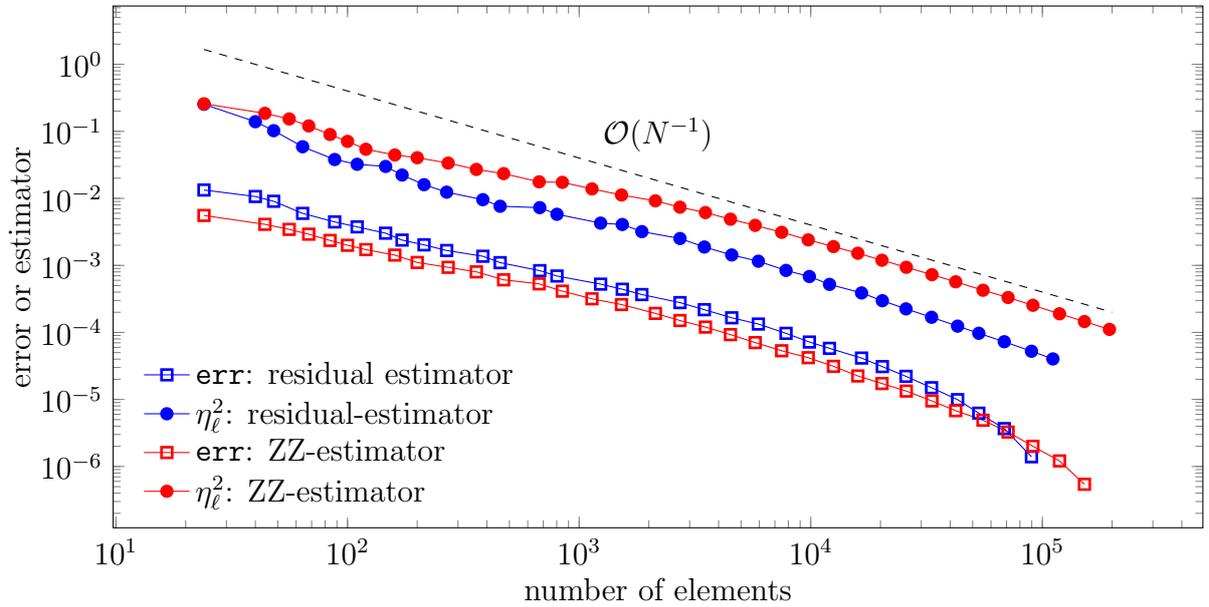

\begin{figure}[h!]
\begin{tikzpicture}
\begin{loglogaxis}[
 width=1.0\textwidth, height=8.5cm,
legend style={at={(0.02,0.02)},
legend cell align={left},
anchor=south west, align=left, draw = none },
xlabel={number of elements},
ylabel={error or estimator},
]

	\addplot+[mark=square,mark size=2pt,mark options={line width=1.0pt},color=red] table
		[x=number_of_elements,y=estimator]{figures/star_theta_02.csv};
	\addlegendentry{$\theta = 0.2$}
	
	\addplot+[mark=*,mark size=2pt,mark options={line width=1.0pt},color=orange] table
		[x=number_of_elements,y=estimator]{figures/star_theta_02.csv};
	\addlegendentry{$\theta = 0.4$}

	\addplot+[mark=triangle,mark size=2pt,mark options={line width=1.0pt},color=blue] table
		[x=number_of_elements,y=estimator]{figures/star_theta_02.csv};
	\addlegendentry{$\theta = 0.6$}
	
	\addplot+[mark=diamond,mark size=2pt,mark options={line width=1.0pt},color=green] table
		[x=number_of_elements,y=estimator]{figures/star_theta_02.csv};
	\addlegendentry{$\theta = 0.8$}
	
	\addplot+[solid,mark=pentagon,mark size=2pt,mark options={line width=1.0pt},color=magenta] table
		[x=number_of_elements,y=estimator]{figures/star_theta_10.csv};
	\addlegendentry{uniform}

	\addplot [black,dashed ] expression [domain=24:100000
	, samples = 10] {50*10^(-1)*x^(-2/3)} node [midway,above,yshift=0.20cm] {$\mathcal{O}(N^{-2/3})$};

	\addplot [black,dashed ] expression [domain=24:100000
		, samples = 10] {10*x^(-1)} node [midway,below,yshift=-0.75cm] {$\mathcal{O}(N^{-1})$};

\end{loglogaxis}
\end{tikzpicture}
\caption{Example~\ref{subsection:star}:Convergence of the error estimator $\eta_\ell^2$ for different values of $0 < \theta <1$ and uniform refinement $\theta = 1$.
Uniform refinement leads to a reduced rate of convergence. }

\label{fig:star:theta}  
\end{figure}
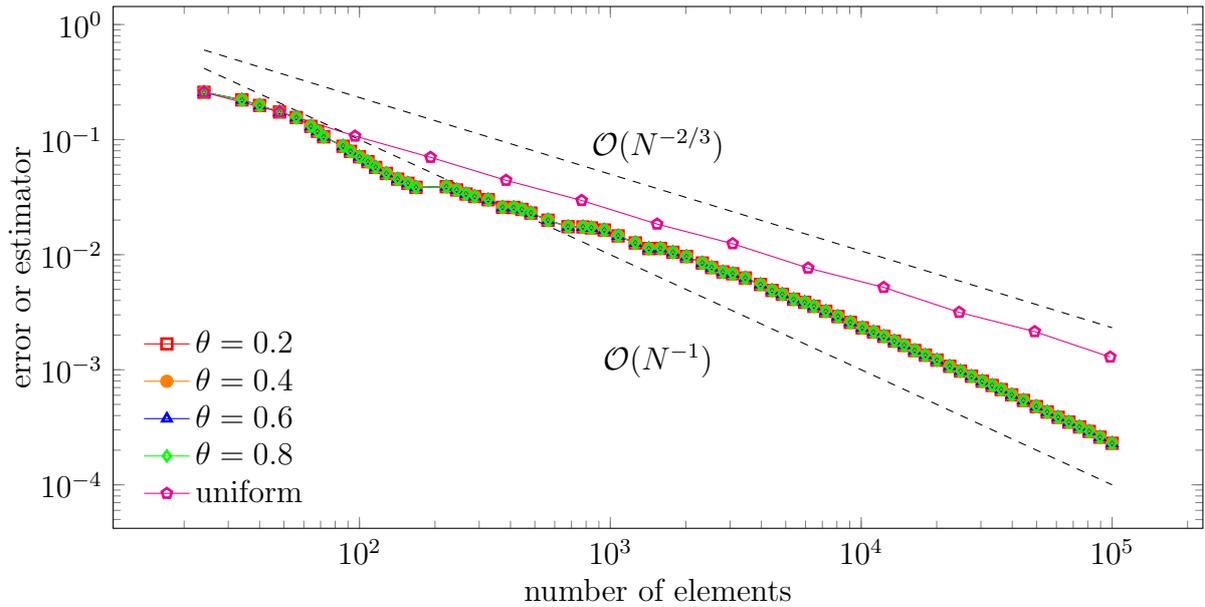

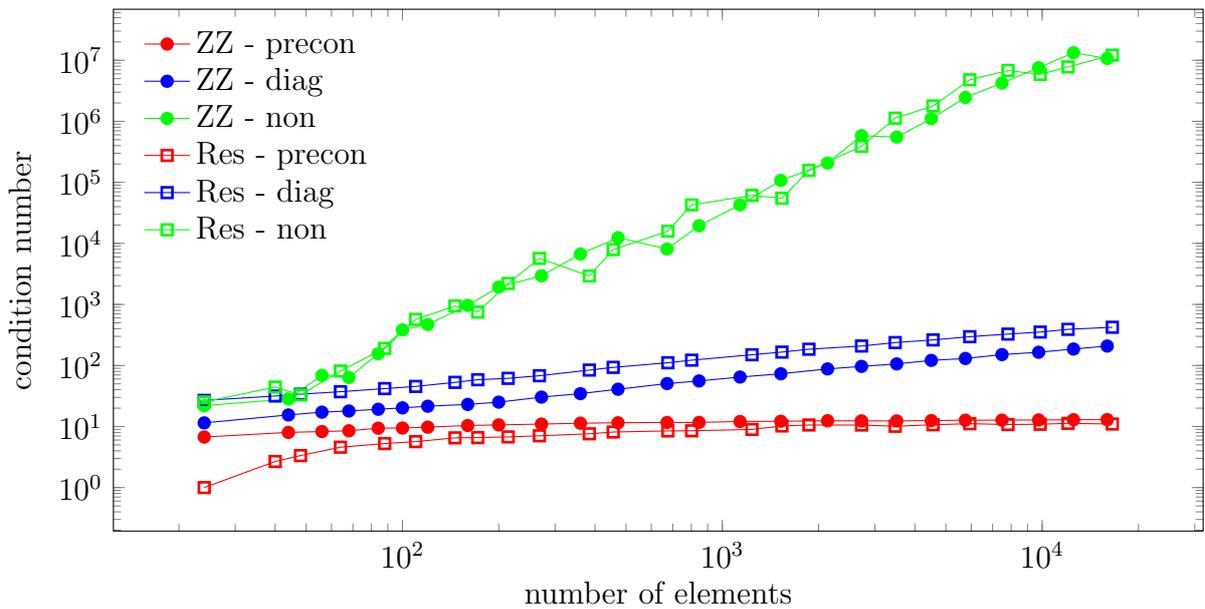
\begin{figure}[h!]
\begin{tikzpicture}
\begin{loglogaxis}[
 width=1.0\textwidth, height=8.5cm,
legend style={at={(0.02,1-0.02)},
legend cell align={left}, anchor=north west, align=left, draw = none },
xlabel={number of elements},
ylabel={condition number},
]
	\addplot+[solid,mark=*,mark size=2pt,mark options={line width=1.0pt},color=red] table
		[x=number_of_elements,y=cond_precon]{figures/star_cond_zz.csv};
	\addlegendentry{ZZ - precon}
	
	\addplot+[solid,mark=*,mark size=2pt,mark options={line width=1.0pt},color=blue] table
		[x=number_of_elements,y=cond_diag]{figures/star_cond_zz.csv};
	\addlegendentry{ZZ - diag}	
	
	\addplot+[solid,mark=*,mark size=2pt,mark options={line width=1.0pt},color=green] table
		[x=number_of_elements,y=cond_non]{figures/star_cond_zz.csv};
	\addlegendentry{ZZ - non}	
	
	\addplot+[solid,mark=square,mark size=2pt,mark options={line width=1.0pt},color=red] table
		[x=number_of_elements,y=cond_precon]{figures/star_cond_res.csv};
	\addlegendentry{Res - precon}
	
	\addplot+[solid,mark=square,mark size=2pt,mark options={line width=1.0pt},color=blue] table
		[x=number_of_elements,y=cond_diag]{figures/star_cond_res.csv};
	\addlegendentry{Res - diag}	
	
	\addplot+[solid,mark=square,mark size=2pt,mark options={line width=1.0pt},color=green] table
		[x=number_of_elements,y=cond_non]{figures/star_cond_res.csv};
	\addlegendentry{Res - non}


\end{loglogaxis}
\end{tikzpicture}
\caption{Example~\ref{subsection:star}: Condition number for different preconditioning strategies for the arising linear systems in Algorithm~\ref{algorithm} (ZZ) with operator preconditioning with the hypersingular operators vs. Algorithm (Res) with multilevel additive Schwarz preconditioning. Additionally, diagonal preconditioning is employed in both cases.  }	
\label{fig:star:iterations}  
\end{figure}

\begin{figure}
	\centering
	\includegraphics[width=0.48\textwidth]{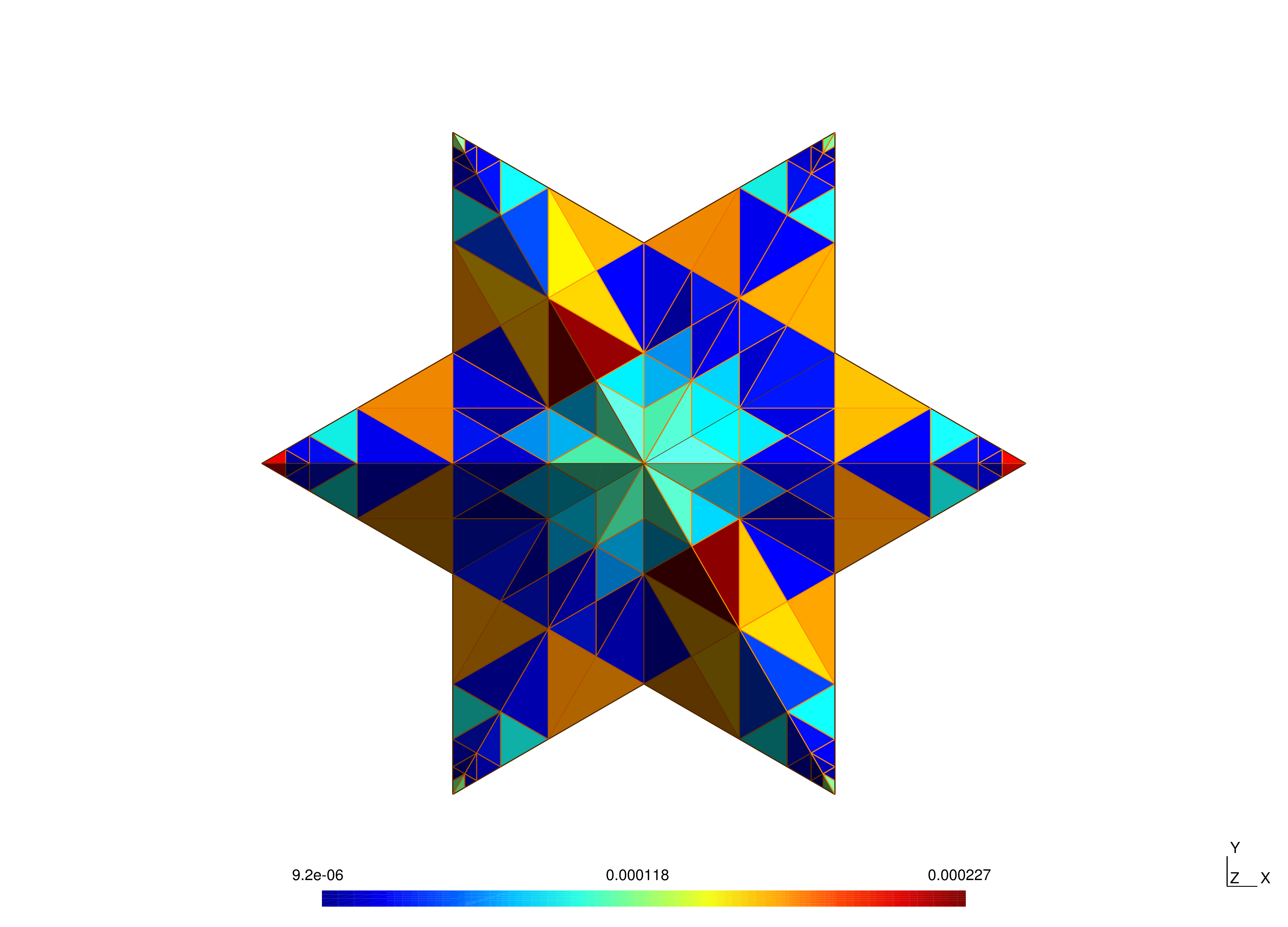}
	\includegraphics[width=0.48\textwidth]{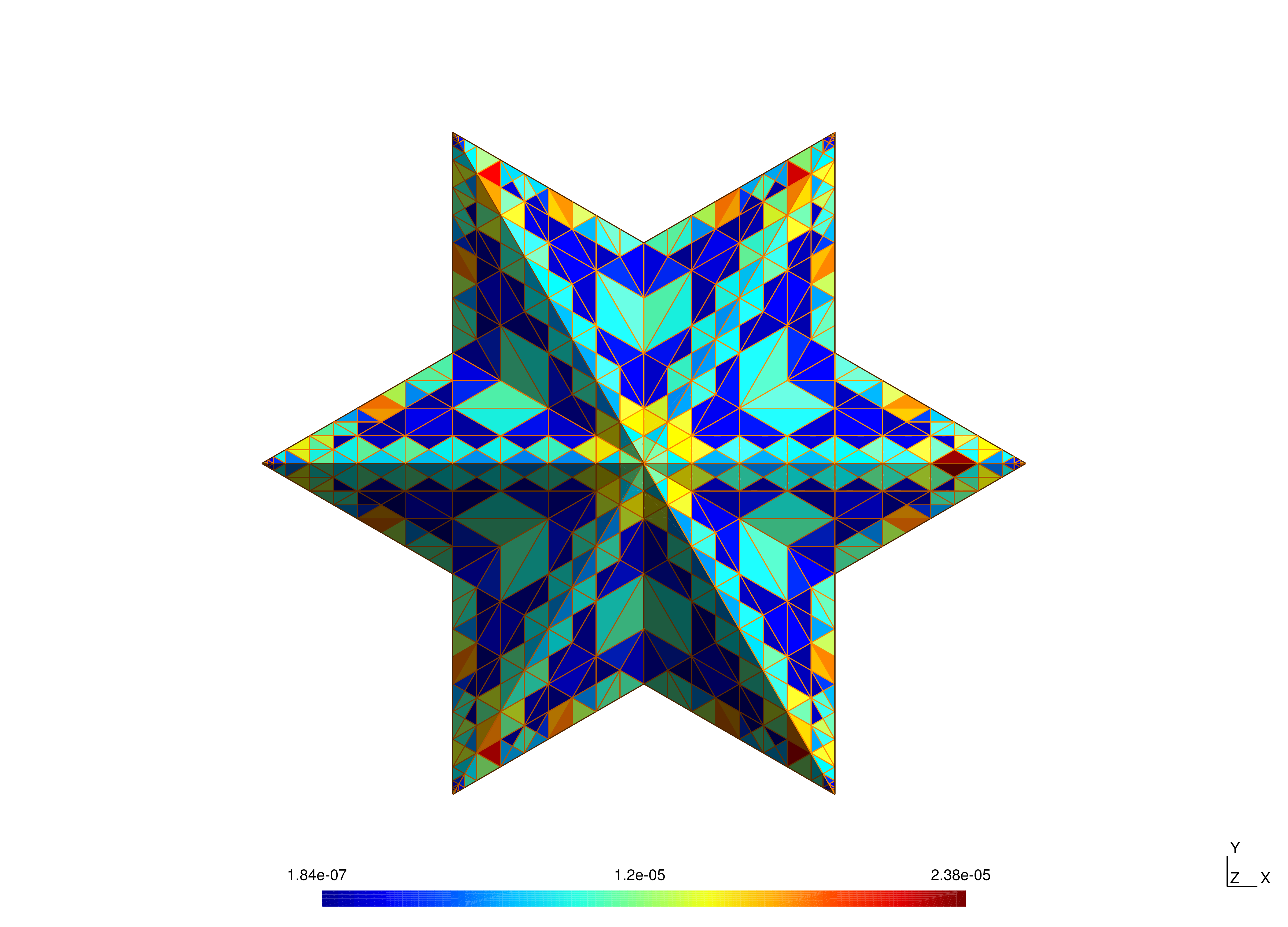}
	\caption{Example~\ref{subsection:star}: Mesh $\TT_\ell$ and distribution of the error estimator $\eta_\ell^2$ for $\ell=10$ (left) with $360$ elements and $\ell = 15$ (right) with $1524$ elements.}
	\label{fig:star:res5_10}	  
\end{figure}

\begin{figure}
	\centering
	\includegraphics[width=0.48\textwidth]{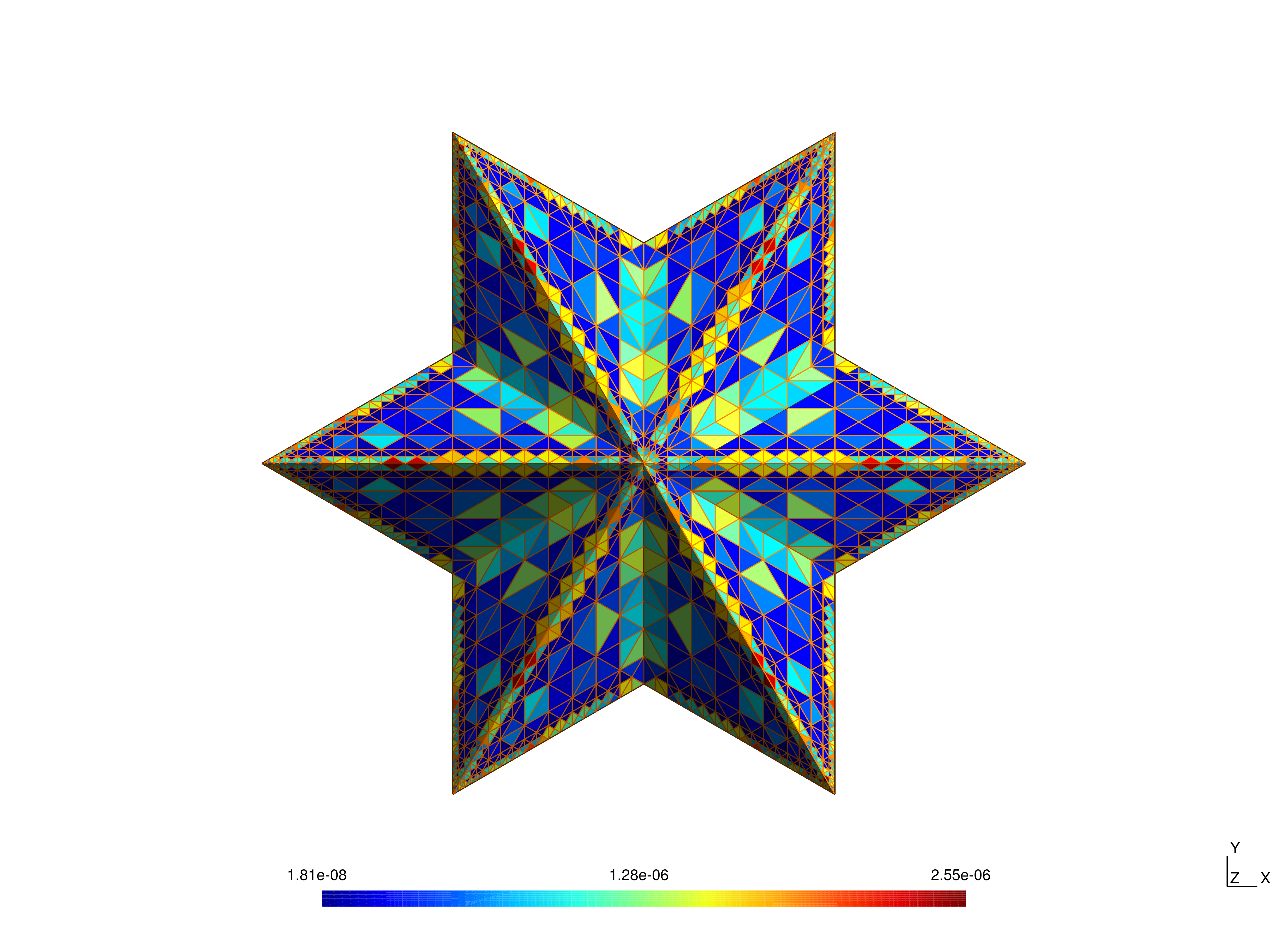}
	\includegraphics[width=0.48\textwidth]{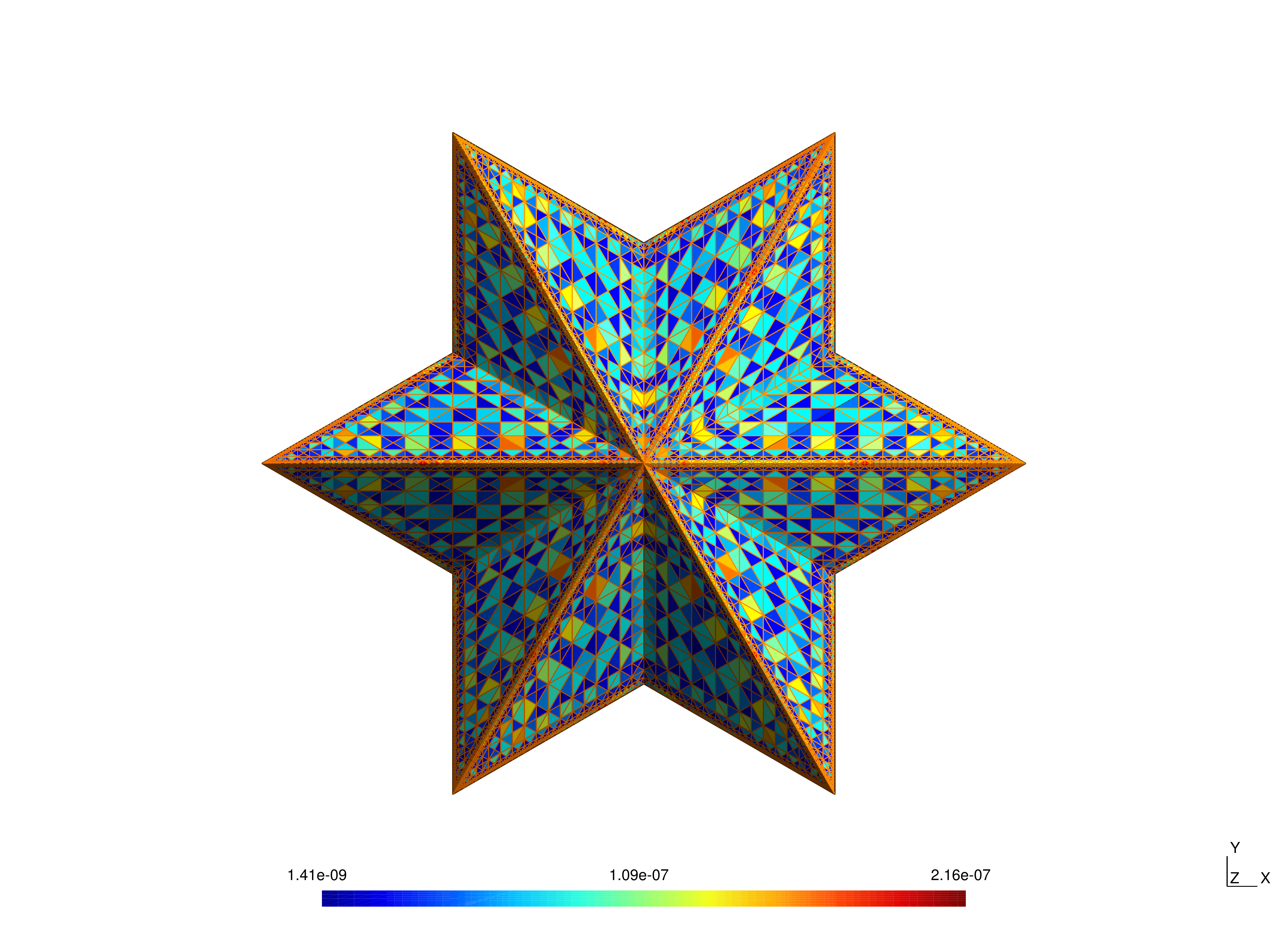}
	\caption{Example~\ref{subsection:star}: Mesh $\TT_\ell$ and distribution of the error estimator $\eta_\ell^2$ for $\ell=20$ (left) with $5754$ elements and $\ell = 25$ (right) with $20304$ elements.}
	\label{fig:star:res20_25}	  
\end{figure}



\appendix

\section{Weighted-residual error estimator}
\label{section:app:residual}

\noindent%
In this section, we briefly sketch the state of the art for adaptive BEM driven by the weighted-residual error estimator.
We consider the model problem~\eqref{eq:strongform} with right-hand side $f \in H^1(\Gamma)$. 
Note that the additional regularity of $f$ (instead of only $f \in H^{1/2}(\Gamma)$) is required for the well-posedness of the weighted-residual error estimator introduced below.
Here, $\TT_\ell$ is a conforming triangulation of $\Gamma$ into plane (closed) surface triangles. The Lax--Milgram lemma guarantees existence and uniqueness of $\Phi_\ell \in \PP^0(\TT_\ell)$ such that
\begin{align}\label{eq:discreteform:residual}
 \dual{V\Phi_\ell}{\Psi_\ell}_\Gamma = \dual{f}{\Psi_\ell}_\Gamma
 \quad \text{for all } \Psi_\ell \in \PP^0(\TT_\ell).
\end{align}
Unlike above, we note that the discrete solution is now computed on the primal triangulation instead of the dual triangulation. With $\nabla$ being the surface gradient, the local contributions of the weighted-residual error estimator read
\begin{align}\label{eq:residual:indicator}
    \mu_\ell(T) := \diam(T)^{1/2} \, \norm{\nabla (f-V \Phi_\ell)}{L^2(T)} \quad  \text{ for all } T \in \TT_\ell,
\end{align}
We note that reliability~\cite{cms01} and efficiency~\cite{invest} hold in the sense of
\begin{align}\label{eq:reliability:residual}
 (\Crel')^{-1} \, \norm{\phi - \Phi_\ell}{H^{-1/2}(\Gamma)}
 \le \mu_\ell := \Big( \sum_{T \in \TT_\ell} \mu_\ell(T)^2 \Big)^{1/2}
 \le \Ceff' \, \norm{h_\ell^{1/2}(\phi-\Phi_\ell)}{L^2(\Gamma)},
\end{align}
where the constants $\Crel', \Ceff' > 0$ depend only on $\Gamma$ and $\gamma$-shape regularity of $\TT_\ell$. As above, $h_\ell$ denotes the local mesh-size function $h_\ell|_T := \diam(T)$ for all $T \in \TT_\ell$. Based on $\mu_\ell$, the standard adaptive algorithm from~\cite{fkmp,gantumur,fkmp:part1} reads as follows:

\begin{algorithm}\label{algorithm:res}
\textbf{Input:} Initial conforming triangulation $\TT_0$ of $\Gamma$ into plane surface triangles, adaptivity parameters $0<\theta\le1$ and $\Cmark\ge1$.
\\
{\bfseries Adaptive loop.} Iterate the following steps~{\rm(i)}--{\rm(iv)} for all $\ell=0,1,2,\dots$:
\begin{itemize}
\item[\rm(i)] Compute the Galerkin solution $\Phi_\ell\in \PP^0(\TT_\ell)$ from~\eqref{eq:discreteform:residual}.

\item[\rm(ii)] Compute the local contributions $\mu_\ell(T)$ from~\eqref{eq:residual:indicator} for all $T\in\TT_\ell$.

\item[\rm(iii)] Determine a set of, up to the multiplicative constant $\Cmark$, minimal cardinality, which satisfies the D\"orfler marking criterion $\theta \, \sum_{T \in \MM_\ell} \mu_\ell(T)^2 \le \mu_\ell^2$.

\item[\rm(iv)] Use NVB and refine all marked elements $T \in \MM_\ell$ to obtain $\TT_{\ell+1} := \refine(\TT_\ell, \MM_\ell)$.
\end{itemize}
\textbf{Output:} Sequence of successively refined triangulations $\TT_\ell$ as well as corresponding Galerkin solutions $\Phi_\ell \in \PP^0(\TT_\ell)$ and weighted-residual error estimators $\mu_\ell$.\qed
\end{algorithm}

The following theorem is the main result from~\cite{fkmp,gantumur,fkmp:part1}:

\begin{theorem}\label{theorem:algorithm:res}
For all $0<\theta\le 1$, there exist constants $\Clin>0$ and $0<\qlin<1$ such that the output of Algorithm~\ref{algorithm} is linearly convergent in the sense of 
\begin{align}\label{eq:thm:linear}
 \eta_{\ell+n} \le \Clin\qlin^n\,\eta_\ell
 \quad\text{for all }\ell,n\in\N_0.
\end{align}
Moreover, there exists a constant $0<\theta_{\rm opt}<1$ such that for all $0<\theta<\theta_{\rm opt}$, the output of Algorithm~\ref{algorithm:res} converges with the optimal algebraic rate in the following sense:
For all $s > 0$, define
\begin{align}\label{eq:class}
 \A_s(\phi) := \sup_{N\ge\#\TT_0}\min_{\substack{\TT_{\rm opt}\in \refine(\TT_0)\\\#\TT_{\rm opt}\le N}}N^s\mu_{\rm opt}
 \in [0,\infty]
\end{align}
where $\refine(\TT_0)$ denotes the set of all possible NVB refinements of $\TT_0$ and $\TT_{\rm opt}$ corresponds to the estimator $\mu_{\rm opt}$.  Then, there exists constants $c_{\rm opt}, C_{\rm opt} > 0$ such that
\begin{align*}
 c_{\rm opt}^{-1} \, \A_s(\phi) \le \sup_{\ell \in \N_0} (\#\TT_\ell)^s \, \mu_\ell \le C_{\rm opt} \, \A_s(\phi),
\end{align*}
i.e., Algorithm~\ref{algorithm:res} asymptotically realizes (and even characterizes) each possible algebraic convergence rate $s>0$ and therefore converges with the optimal rate.
The constants $\Clin$ and $\qlin$ depend only on $\Gamma$, $\gamma$-shape regularity for $\TT_0$, and $\theta$.
While $c_{\rm opt}$ is generic, $C_{\rm opt}$ depends only on $\TT_0$, $\Clin$, $\qlin$, $\Cmark$, and $s$.
\qed
\end{theorem}

\begin{remark}
	The recent work~\cite{abem+solve} proves that Theorem~\ref{theorem:algorithm:res} remains valid, if~\eqref{eq:discreteform:residual} is solved
	inexactly by PCG with optimal additive Schwarz preconditioner. With the PCG iterates $\Phi_{\ell,k} \approx \Phi_\ell$, the iterative solver is stopped 
	if $\enorm{\Phi_{\ell,k-1} - \Phi_{\ell,k} } \leq \lambda \mu_\ell(\Phi_{\ell,k})$, where the residual error estimator is evaluated at $\Phi_{\ell,k}$ instead 
	of the exact Galerkin solution. Then, linear convergence~\eqref{eq:thm:linear} holds for any $\lambda >0$, while optimal rates require $\lambda \ll 1$ to be sufficiently small.
\end{remark}

To empirically measure the performance of Algorithm~\ref{algorithm} in the numerical experiments of Section~\ref{section:computations}, we consider the output of Algorithm~\ref{algorithm:res} for $f = 1$ as benchmark. 

\begin{corollary}
Consider the output of Algorithm~\ref{algorithm:res} for $f = 1$ and arbitrary $0 < \theta \le 1$ and $\Cmark \ge 1$. Define the discrete capacity
\begin{align}
 \capacity'_\ell(\Omega) := \frac{1}{4\pi} \, \dual{\Phi_\ell}{1}_\Gamma.
\end{align}
Then, there holds~\eqref{eq:prop:cap1} as well as monotonicity $0 < \capacity'_\ell(\Omega) \le \capacity'_{\ell + 1} \le \capacity(\Omega)$ for all $\ell \in \N_0$ and, in particular, convergence
\begin{align}\label{eq1:corollary}
 0 \le \capacity(\Omega) - \capacity'_\ell(\Omega) \le C_{\rm cap}\, \mu_\ell^2 \to 0 
 \quad \text{as } \ell \to \infty.
\end{align} 
Moreover, if $0 < \theta < \theta_{\rm opt}$ is sufficiently small in the sense of Theorem~\ref{theorem:algorithm:res}, it follows that
\begin{align}\label{eq2:corollary}
 0 \le \capacity(\Omega) - \capacity'_\ell(\Omega) \le C'_{\rm cap} \, \A_s(\phi)^2 \, (\#\TT_\ell)^{-2s}
 \quad \text{for all } \ell \in \N_0.
\end{align}
The constant $C_{\rm cap}$ depends only on $\Gamma$ and $\gamma$-shape regularity of $\TT_\ell$, and $C'_{\rm cap}:= C_{\rm cap} \Copt^2$ with $\Copt$ from Theorem~\ref{theorem:algorithm:res}.
\end{corollary}

\begin{proof}
Arguing as in the proof of Proposition~\ref{proposition:capacity}, we see that 
\begin{align*}
 0 < \capacity'_\ell(\Omega) \le \capacity(\Omega) 
 \quad \text{as well as} \quad
 0 \le \capacity(\Omega) - \capacity'_\ell(\Omega)
 \simeq \norm{\phi - \Phi_\ell}{H^{-1/2}(\Gamma)}^2.
\end{align*}
Therefore,~\eqref{eq1:corollary}--\eqref{eq2:corollary} are immediate consequences of Theorem~\ref{theorem:algorithm:res} and reliability~\eqref{eq:reliability:residual}. Overall, it only remains to prove monotonicity $\capacity'_\ell(\Omega) \le \capacity'_{\ell + 1}(\Omega)$ for all $\ell \in \N_0$: 
Note that $\PP^0(\TT_\ell) \subseteq \PP^0(\TT_{\ell+1})$, which fails for the sequence of dual triangulations, i.e., $\PP^0(\TT_\ell^{\rm dual}) \not \subseteq \PP^0(\TT_{\ell+1}^{\rm dual})$. With this additional nestedness, it holds that $\Phi_{\ell+1}-\Phi_\ell\in \PP^0(\TT_{\ell+1})$. Therefore, there holds the discrete Galerkin orthogonality 
$\dual{V(\Phi_{\ell+1}-\Phi_\ell)}{\Psi_{\ell+1} }_\Gamma = 0$ for all $\Psi_{\ell+1} \in  \PP^0(\TT_{\ell+1} )$. 
This implies the Pythagoras theorem
\begin{align*}
 \enorm{\phi-\Phi_\ell}^2 = \enorm{\phi-\Phi_{\ell+1}}^2 + \enorm{\Phi_{\ell+1}-\Phi_\ell}^2
\end{align*}
with respect to the energy norm.
Arguing as for the proof of~\eqref{eq:prop:cap1}, we thus see that
\begin{align*}
 4 \pi \, \big( \capacity_{\ell+1}(\Omega) - \capacity_\ell(\Omega) \big)
 = \enorm{\Phi_{\ell+1}-\Phi_\ell}^2 \ge 0.
\end{align*}
This concludes the proof.
\end{proof}

Finally, we refer to~\cite{fkmp:part2} for the extension of Theorem~\ref{theorem:algorithm:res}
to the hypersingular operator associated with the Laplace problem and to~\cite{bem_helmholtz} for optimality in the case of the Helmholtz equation.
The work~\cite{pointabem} gives a first mathematical proof of optimal algebraic convergence rates for the point-wise approximation of 
the PDE-solution via BEM and potential operators. 
Moreover,~\cite{goafem} generalizes the latter approach to general goal-oriented adaptivity with linear goal functional.


\section{Implementation with Bempp}
\label{section:app:code}

\definecolor{keywords}{RGB}{255,0,90}
\definecolor{my_blue}{RGB}{0,0,113}
\definecolor{my_red}{RGB}{255,0,0}
\definecolor{my_green}{RGB}{0,150,0}

\lstset{language=Python, 
	basicstyle=\ttfamily\tiny, 
	numbers=left,
	keywordstyle=\color{my_blue},
	commentstyle=\color{red},
	stringstyle=\color{my_green},
	showstringspaces=false,
	identifierstyle=\color{black},
	procnamekeys={def,class}}

The following code demonstrates a possible implementation of Algorithm~\ref{algorithm}
in Python with the Bempp library. The demonstration code can also be downloaded from 
the Bempp homepage \href{https://bempp.com}{https://bempp.com}.

\begin{lstlisting}
import bempp.api
import numpy as np

max_el = 600000 #max number elements
step_counter = 0
theta = 0.5 #marking parameter
grid = bempp.api.shapes.cube(h=0.1) #import initial grid
#initialize function for the right-hand side
def one_fun(x, n, domain_index, res):
    res[0] = -1

##Adaptive loop
while (grid.leaf_view.entity_count(0) <= max_el):
    step_counter += 1   
    bary_grid = grid.barycentric_grid() #barycentric refinement

    ##initialize function spaces
    const_space = bempp.api.function_space(grid, 'DP', 0)
    lin_space = bempp.api.function_space(grid, 'P',1)
    bary_space = bempp.api.function_space(bary_grid, 'DP', 0)
    bary_lin_space = bempp.api.function_space(bary_grid, 'DP',1) 
    
    ##initialize operators
    base_slp = bempp.api.operators.boundary.laplace.single_layer(bary_space, bary_space, bary_space)
    slp, hyp = bempp.api.operators.boundary.laplace.single_layer_and_hypersingular_pair(
            grid, spaces='dual', base_slp=base_slp,stabilization_factor=1)  

    rank_one_op = bempp.api.RankOneBoundaryOperator(hyp.domain, hyp.range, hyp.dual_to_range)
    hyp_regularized = hyp + rank_one_op
    
    ##initialize right-hand side
    rhs_fun = bempp.api.GridFunction(slp.range, fun=one_fun)

    ##initialize preconditioned lhs
    lhs = slp * hyp_regularized
    discrete_op = lhs.strong_form()

    ##solve via GMRES
    from scipy.sparse.linalg import gmres
    number_of_iterations = 0	
    def callback(x):
        global number_of_iterations
        number_of_iterations += 1
    
    phi_vec, info = gmres(discrete_op, rhs_fun.coefficients,tol=1e-10, callback=callback)
    phi_fun = hyp_regularized * bempp.api.GridFunction(hyp.domain, coefficients=phi_vec)

    print("Number of iterations: {0}".format(number_of_iterations))

    ##compute the capacity
    capacity=-phi_fun.integrate()[0,0]
    print("The capacity is {0}.".format(capacity))

    ##compute the error estimator
    bary_map = grid.barycentric_descendents_map()
    bary_index_set = bary_grid.leaf_view.index_set()

    #create sorted list of elements
    index_set = grid.leaf_view.index_set() 
    elements = grid.leaf_view.entity_count(0) * [None]
    for element in grid.leaf_view.entity_iterator(0):
        elements[index_set.entity_index(element)] = element

    local_est_bary = np.zeros(bary_grid.leaf_view.entity_count(0), dtype='float64')
    local_est = np.zeros(grid.leaf_view.entity_count(0), dtype='float64')
    new_space = bempp.api.function_space(grid,"B-P",1)	    

    map_dual_to_bary = bempp.api.operators.boundary.sparse.identity(
            phi_fun.space,bary_lin_space,bary_lin_space)
    map_lin_to_bary_lin =  bempp.api.operators.boundary.sparse.identity(
            new_space,bary_lin_space,bary_lin_space)
    
    #interpret function values on dual elements as node-values in S^1
    Iphi_coefficients = phi_fun.coefficients 
    Iphi = bempp.api.GridFunction(new_space, coefficients=Iphi_coefficients)

    #lift both functions to P^1(T_l^bary)
    phi = map_dual_to_bary * phi_fun
    Iphi_lin_bary = map_lin_to_bary_lin * Iphi

    zz_diff = phi - Iphi_lin_bary
    
    #compute the L^2-norm in each barycentric element
    for element in bary_grid.leaf_view.entity_iterator(0):
        index = bary_index_set.entity_index(element)
        local_est_bary [index] = zz_diff.l2_norm(element)**2 
  
    #add up all barycentric contributions for each primal element 
    for m in range(bary_map.shape[0]):
        element = grid.leaf_view.element_from_index(m)
        mesh_size = element.geometry.volume
        for n in range(bary_map.shape[1]):
            local_est[m] += local_est_bary[bary_map[m, n]]

        local_est[m] = local_est[m]*mesh_size**(1./2) #scale with primal mesh-size
        
    total_est = np.sum(local_est)
    
    print("Squared ZZ-Estimator: {0}".format(total_est))

    ##Doerfler-Marking and refining
    idx = np.argsort((-1)*local_est)
    ind_on_marked =0 
    counter = 0 

    while (ind_on_marked < theta * total_est):
        ind_on_marked += local_est[idx[counter]]
        grid.mark(elements[idx[counter]])
        counter += 1
   
    grid = grid.refine() # refine grid
    print("New number of elements {0}".format(grid.leaf_view.entity_count(0)))


\end{lstlisting}

\bigskip



\bibliographystyle{alpha}
\bibliography{literature,literature_dpr}

\newcommand{\etalchar}[1]{$^{#1}$}
\begin{thebibliography}{vtWGBA15}

\bibitem[AFF{\etalchar{+}}17]{invest}
Markus Aurada, Michael Feischl, Thomas F{\"u}hrer, Michael Karkulik, J.~Markus
  Melenk, and Dirk Praetorius.
\newblock Local inverse estimates for non-local boundary integral operators.
\newblock {\em Math. Comp.}, 86(308):2651--2686, 2017.

\bibitem[BBHP19]{bem_helmholtz}
Alex Bespalov, Timo Betcke, Alexander Haberl, and Dirk Praetorius.
\newblock Adaptive {BEM} with optimal convergence rates for the {H}elmholtz
  equation.
\newblock {\em Comput. Methods Appl. Mech. Engrg.}, 346:260--287, 2019.

\bibitem[BC02]{bc2002}
S\"{o}ren Bartels and Carsten Carstensen.
\newblock Each averaging technique yields reliable a posteriori error control
  in {FEM} on unstructured grids. {I}. {L}ow order conforming, nonconforming,
  and mixed {FEM}.
\newblock {\em Math. Comp.}, 71(239):945--969, 2002.

\bibitem[Beb08]{bebendorf2008}
Mario Bebendorf.
\newblock {\em Hierarchical matrices}, volume~63 of {\em Lecture Notes in
  Computational Science and Engineering}.
\newblock Springer-Verlag, Berlin, 2008.

\bibitem[B{\"o}r10]{boerm2010}
Steffen B{\"o}rm.
\newblock {\em Efficient numerical methods for non-local operators:
  ${\mathcal{H}}^2$-matrix compression, algorithms and analysis}.
\newblock European Mathematical Society (EMS), Z\"{u}rich, 2010.

\bibitem[BT03]{banjai03}
Lehel Banjai and Lloyd~N. Trefethen.
\newblock A multipole method for {S}chwarz-{C}hristoffel mapping of polygons
  with thousands of sides.
\newblock {\em SIAM J. Sci. Comput.}, 25(3):1042--1065, 2003.

\bibitem[CMS01]{cms01}
Carsten Carstensen, Matthias Maischak, and Ernst~P. Stephan.
\newblock A posteriori error estimate and {$h$}-adaptive algorithm on surfaces
  for {S}ymm's integral equation.
\newblock {\em Numer. Math.}, 90(2):197--213, 2001.

\bibitem[D{\"o}r96]{doerfler96}
Willy D{\"o}rfler.
\newblock A convergent adaptive algorithm for {P}oisson's equation.
\newblock {\em SIAM J. Numer. Anal.}, 33(3):1106--1124, 1996.

\bibitem[DT02]{driscoll02}
Tobin~A. Driscoll and Lloyd~N. Trefethen.
\newblock {\em Schwarz-{C}hristoffel mapping}.
\newblock Cambridge University Press, Cambridge, 2002.

\bibitem[FFH{\etalchar{+}}15]{FFHKP15}
Michael Feischl, Thomas F{\"u}hrer, Norbert Heuer, Michael Karkulik, and Dirk
  Praetorius.
\newblock Adaptive boundary element methods.
\newblock {\em Archives of Computational Methods in Engineering},
  22(3):309--389, Jul 2015.

\bibitem[FFK{\etalchar{+}}14]{fkmp:part1}
Michael Feischl, Thomas F{\"u}hrer, Michael Karkulik, Jens~Markus Melenk, and
  Dirk Praetorius.
\newblock Quasi-optimal convergence rates for adaptive boundary element methods
  with data approximation, {P}art {I}: {W}eakly-singular integral equation.
\newblock {\em Calcolo}, 51(4):531--562, 2014.

\bibitem[FFK{\etalchar{+}}15]{fkmp:part2}
Michael Feischl, Thomas F{\"u}hrer, Michael Karkulik, J.~Markus Melenk, and
  Dirk Praetorius.
\newblock Quasi-optimal convergence rates for adaptive boundary element methods
  with data approximation. {P}art {II}: {H}yper-singular integral equation.
\newblock {\em Electron. Trans. Numer. Anal.}, 44:153--176, 2015.

\bibitem[FFKP14]{ffkp14}
Michael Feischl, Thomas F\"{u}hrer, Michael Karkulik, and Dirk Praetorius.
\newblock {ZZ}-type a posteriori error estimators for adaptive boundary element
  methods on a curve.
\newblock {\em Eng. Anal. Bound. Elem.}, 38:49--60, 2014.

\bibitem[FGH{\etalchar{+}}16]{pointabem}
Michael Feischl, Gregor Gantner, Alexander Haberl, Dirk Praetorius, and Thomas
  F\"uhrer.
\newblock Adaptive boundary element methods for optimal convergence of point
  errors.
\newblock {\em Numer. Math.}, 132(3):541--567, 2016.

\bibitem[FHPS19]{abem+solve}
Thomas F\"uhrer, Alexander Haberl, Dirk Praetorius, and Stefan Schimanko.
\newblock Adaptive {B}{E}{M} with inexact {P}{C}{G} solver yields almost
  optimal computational costs.
\newblock {\em Numer. Math.}, 141:967--1008, 2019.

\bibitem[FKMP13]{fkmp}
Michael Feischl, Michael Karkulik, {Jens Markus} Melenk, and Dirk Praetorius.
\newblock Quasi-optimal convergence rate for an adaptive boundary element
  method.
\newblock {\em SIAM J. Numer. Anal.}, 51(2):1327--1348, 2013.

\bibitem[FPvdZ16]{goafem}
Michael Feischl, Dirk Praetorius, and Kristoffer~G. van~der Zee.
\newblock An abstract analysis of optimal goal-oriented adaptivity.
\newblock {\em SIAM J. Numer. Anal.}, 54(3):1423--1448, 2016.

\bibitem[Gan13]{gantumur}
Tsogtgerel Gantumur.
\newblock Adaptive boundary element methods with convergence rates.
\newblock {\em Numer. Math.}, 124(3):471--516, 2013.

\bibitem[GBB{\etalchar{+}}15]{bempp2}
Samuel~P. Groth, Anthony~J. Baran, Timo Betcke, Stephan Havemann, and Wojciech
  \'Smigaj.
\newblock The boundary element method for light scattering by ice crystals and
  its implementation in bem++.
\newblock {\em J. Quant. Spectrosc. Radiat. Transf.}, 167(Supplement C):40 --
  52, 2015.

\bibitem[GGMR09]{gr2009}
Leslie Greengard, Denis Gueyffier, Per-Gunnar Martinsson, and Vladimir Rokhlin.
\newblock Fast direct solvers for integral equations in complex
  three-dimensional domains.
\newblock {\em Acta Numer.}, 18:243--275, 2009.

\bibitem[GM06]{gm06}
Ivan~G. Graham and William McLean.
\newblock Anisotropic mesh refinement: the conditioning of {G}alerkin boundary
  element matrices and simple preconditioners.
\newblock {\em SIAM J. Numer. Anal.}, 44(4):1487--1513, 2006.

\bibitem[GR97]{gr1997}
Leslie Greengard and Vladimir Rokhlin.
\newblock A new version of the fast multipole method for the {L}aplace equation
  in three dimensions.
\newblock In {\em Acta numerica, 1997}, volume~6 of {\em Acta Numer.}, pages
  229--269. Cambridge Univ. Press, Cambridge, 1997.

\bibitem[GS18]{eps_gwinner}
Joachim Gwinner and Ernst~Peter Stephan.
\newblock {\em Advanced boundary element methods}.
\newblock Springer, Cham, 2018.

\bibitem[Hac15]{hackbusch2015}
Wolfgang Hackbusch.
\newblock {\em Hierarchical matrices: algorithms and analysis}.
\newblock Springer, Heidelberg, 2015.

\bibitem[Hip06]{Hiptmair2006}
Ralf Hiptmair.
\newblock {Operator Preconditioning}.
\newblock {\em Comput. Math. Appl.}, 52(5):699--706, 2006.

\bibitem[HJHUT14]{hju2014}
Ralf Hiptmair, Carlos Jerez-Hanckes, and Carolina Urz\'{u}a-Torres.
\newblock Mesh-independent operator preconditioning for boundary elements on
  open curves.
\newblock {\em SIAM J. Numer. Anal.}, 52(5):2295--2314, 2014.

\bibitem[HMW10]{HMW10}
Chi-Ok Hwang, Michael Mascagni, and Taeyoung Won.
\newblock Monte {C}arlo methods for computing the capacitance of the unit cube.
\newblock {\em Math. Comput. Simulation}, 80(6):1089--1095, 2010.

\bibitem[HP13]{hp13}
Johan Helsing and Karl-Mikael Perfekt.
\newblock On the polarizability and capacitance of the cube.
\newblock {\em Appl. Comput. Harmon. Anal.}, 34(3):445--468, 2013.

\bibitem[HW08]{hsiao-wendland}
George~C. Hsiao and Wolfgang~L. Wendland.
\newblock {\em Boundary integral equations}.
\newblock Springer, Berlin, 2008.

\bibitem[Kel67]{kellogg}
Oliver~Dimon Kellogg.
\newblock {\em Foundations of potential theory}.
\newblock Reprint from the first edition of 1929. Springer, Berlin, 1967.

\bibitem[KPP13]{kpp13}
Michael Karkulik, David Pavlicek, and Dirk Praetorius.
\newblock On 2{D} newest vertex bisection: optimality of mesh-closure and
  {$H^1$}-stability of {$L_2$}-projection.
\newblock {\em Constr. Approx.}, 38(2):213--234, 2013.

\bibitem[McL00]{mclean}
William McLean.
\newblock {\em Strongly elliptic systems and boundary integral equations}.
\newblock Cambridge University Press, Cambridge, 2000.

\bibitem[OSW06]{osw2006}
G{\"u}nther Of, Olaf Steinbach, and Wolfgang~L. Wendland.
\newblock The fast multipole method for the symmetric boundary integral
  formulation.
\newblock {\em IMA J. Numer. Anal.}, 26(2):272--296, 2006.

\bibitem[P{\'o}l47]{polya47}
George P{\'o}lya.
\newblock Estimating electrostatic capacity.
\newblock {\em Amer. Math. Monthly}, 54(4):201--206, 1947.

\bibitem[Rea97]{Read97}
{Frank H.} Read.
\newblock Improved extrapolation technique in the boundary element method to
  find the capacitances of the unit square and cube.
\newblock {\em J. Comp. Phys.}, 133(1):1 -- 5, 1997.

\bibitem[Rod94]{rodriguez1994}
Rodolfo Rodriguez.
\newblock Some remarks on zienkiewicz--zhu estimator.
\newblock {\em Numer. Methods Partial Differential Equations}, 10(5):625--635,
  1994.

\bibitem[SBA{\etalchar{+}}15]{bempp}
Wojciech \'Smigaj, Timo Betcke, Simon Arridge, Joel Phillips, and Martin
  Schweiger.
\newblock Solving boundary integral problems with {BEM}++.
\newblock {\em ACM Trans. Math. Software}, 41(2):Art. 6, 40, 2015.

\bibitem[SS11]{sauter-schwab}
Stefan~A. Sauter and Christoph Schwab.
\newblock {\em Boundary element methods}.
\newblock Springer, Berlin, 2011.

\bibitem[Ste07]{stevenson07}
Rob Stevenson.
\newblock Optimality of a standard adaptive finite element method.
\newblock {\em Found. Comput. Math.}, 7(2):245--269, 2007.

\bibitem[Ste08]{steinbach}
Olaf Steinbach.
\newblock {\em {Numerical approximation methods for elliptic boundary value
  problems}}.
\newblock Springer, New York, 2008.

\bibitem[SW98]{StWe98}
Olaf Steinbach and Wolfgang~L. Wendland.
\newblock {The construction of some efficient preconditioners in the boundary
  element method}.
\newblock {\em Adv. Comput. Math.}, 9(1-2):191--216, 1998.

\bibitem[vtWGBA15]{bempp3}
Elwin van~'t Wout, Pierre Gélat, Timo Betcke, and Simon Arridge.
\newblock A fast boundary element method for the scattering analysis of
  high-intensity focused ultrasound.
\newblock {\em J. Acoust. Soc. Am.}, 138(5):2726--2737, 2015.

\bibitem[ZZ87]{zz1987}
Oleg Zienkiewicz and {Jian Z.} Zhu.
\newblock A simple error estimator and adaptive procedure for practical
  engineering analysis.
\newblock {\em Int. J. Numer. Methods Eng.}, 24(2):337--357, 1987.

\end{thebibliography}

\end{document}